\documentclass[11 pt,leqno]{article}
\title{
Peak reduction and finite presentations for automorphism groups of right-angled Artin groups
}
\author{Matthew B. Day }
\date{July 29, 2008}
\usepackage{slashbox}
\usepackage{amsmath}
\usepackage{amsfonts}
\usepackage{amssymb}
\usepackage{amsthm}
\usepackage{latexsym}
\usepackage{epsfig}
\usepackage{xypic}
\usepackage{verbatim}
\usepackage{mathrsfs}

\theoremstyle{plain} \newtheorem{theorem}{Theorem}[section]
\theoremstyle{plain} \newtheorem{proposition}[theorem]{Proposition}
\theoremstyle{plain} \newtheorem{lemma}[theorem]{Lemma}
\theoremstyle{plain} \newtheorem{sublemma}[theorem]{Sublemma}
\theoremstyle{plain} \newtheorem{corollary}[theorem]{Corollary}
\theoremstyle{plain} \newtheorem{claim}[theorem]{Claim}
\theoremstyle{plain} 
\theoremstyle{plain} 
\theoremstyle{plain} \newtheorem{conjecture}[theorem]{Conjecture}
\theoremstyle{plain} 
\theoremstyle{remark} \newtheorem{example}[theorem]{Example}
\theoremstyle{definition} \newtheorem{remark}[theorem]{Remark}
\theoremstyle{definition} \newtheorem{definition}[theorem]{Definition}
\theoremstyle{definition} \newtheorem{theorem-definition}[theorem]{Theorem-Definition}
\theoremstyle{definition} 
\theoremstyle{plain} \newtheorem{maintheorem}{Theorem}
\theoremstyle{plain} 
\theoremstyle{plain} \newtheorem{mainproposition}[maintheorem]{Proposition}
\theoremstyle{plain} 

\numberwithin{equation}{section}

\newcommand\Aut{\mathrm{Aut}\,}

\newcommand\Out{\mathrm{Out}\,}
\newcommand\Inn{\mathrm{Inn}\,}

\newcommand\GL{\mathrm{GL}}
\newcommand\SL{\mathrm{SL}}

\newcommand\Z{\mbox{$\mathbb{Z}$}}

\newcommand\into\hookrightarrow

\newcommand\isomarrow{\stackrel{\cong}{\longrightarrow}}
\def\co{\colon\thinspace}

\newcommand\adj[2]{\mathrm e( #1,#2)}
\newcommand\AAG{\Aut A_\Gamma}
\newcommand\AAGo{\Aut^0 A_\Gamma}
\newcommand\OAG{\Out A_\Gamma}

\newcommand\supp{\mathrm{supp\,}}

\newcommand\lk{\mathrm{lk}}
\newcommand\st{\mathrm{st}}
\newcommand\pg[1]{\mathrm v(#1)}
\newcommand\aco[3]{\langle #1,#2\rangle_{#3}}
\newcommand\lkl[1]{\lk_L(#1)}
\newcommand\stl[1]{\st_L(#1)}
\newcommand\OLS{\Omega_\ell\cup\Omega_s}

\begin{document}
\maketitle

\begin{abstract}
We generalize the peak-reduction algorithm (Whitehead's theorem) for free groups to a theorem about a general right-angled Artin group $A_\Gamma$. 
As an application, we find a finite presentation for the automorphism group $\mathrm{Aut}\,A_\Gamma$ that generalizes McCool's presentation for the automorphism group of a finite rank free group.
We also give consider a stronger generalization of peak-reduction, giving a counterexample and proving a special case.
\end{abstract}

\section{Introduction}
\subsection{Background}
Let $\Gamma$ be a graph on $n$ vertices, with vertex set $X$ and adjacency relation denoted by $\adj{-}{-}$.
Let $A_\Gamma$ be the \emph{right-angled Artin group of $\Gamma$}, defined by
\[A_\Gamma:= \langle X |R_\Gamma \rangle\]
where the relations are $R_\Gamma=\{[x,y]|\text{ $x,y\in X$ and $\adj{x}{y}$}\}$ (we use the convention that $[x,y]=xyx^{-1}y^{-1}$).
If $\Gamma$ is the edgeless graph ($n$ vertices and no edges), then $A_\Gamma$ is the free group $F_n$ on $n$ generators.
If $\Gamma$ is the complete graph, then $A_\Gamma$ is the free abelian group $\Z^n$.
So in a sense, the group $A_\Gamma$ interpolates between free groups and free abelian groups as we vary $\Gamma$.
Similarly, automorphism group $\AAG$ interpolates between $\Aut F_n$ and the integral general linear group $\GL(n,\Z)$.
In this paper, we develop a framework for understanding $\AAG$ in which ideas from the study of linear groups and ideas from the study of $\Aut F_n$ can both be applied.
We also give a finite presentation for $\AAG$.

Whitehead's 1936 theorem~(\cite{wh}, Theorem 2) is a result about automorphism groups of free groups with important applications; peak reduction is an algorithmic approach used by Rapaport~\cite{rap}, Higgins--Lyndon~\cite{hl} and others to reprove and extend this theorem.
Whitehead's theorem states that there is a finite generating set for $\Aut F_n$ that has a special property concerning factorizations of elements of $\Aut F_n$ and the lengths of elements of $F_n$.
One corollary of this theorem is that for any $k$--tuple $W$ of elements of $F_n$, the stabilizer $(\Aut F_n)_{W}$ is finitely generated; in fact, McCool~\cite{mccool} used peak-reduction methods to prove that $(\Aut F_n)_{W}$ is finitely presented.
Another corollary is that there is an algorithm that determines whether two $k$--tuples of elements of $F_n$ are in the same orbit under the action of $\Aut F_n$ (see Lyndon--Schupp~\cite{ls}, Chapter~I, Proposition~4.19).
Peak-reduction methods were also used by Culler--Vogtmann~\cite{cullervogtmann} to study the structure of the outer space of $F_n$.

We proceed to define the notion of a peak-reduced factorization.
Define the length of a conjugacy class of $A_\Gamma$ to be the minimum of the lengths of its representative elements (with respect to $X$), and define the length of a $k$--tuple of conjugacy classes of $A_\Gamma$ to be the sum of the lengths of its elements (for any $k\geq 1$).
For $W$ a $k$--tuple of conjugacy classes in $A_\Gamma$, we say that a string $\alpha_m\cdots\alpha_1$ of elements of $\Aut A_\Gamma$ is \emph{peak-reduced} with respect to $W$ if for each $i=1,\ldots,m-1$, we do not have both
\begin{align*}
|(\alpha_{i+1}\cdots\alpha_1)\cdot W |\leq|(\alpha_i\cdots\alpha_1)\cdot W|\\
\vspace{2pt}
\tag*{\text{and}}|(\alpha_i\cdots\alpha_1)\cdot W|\leq|(\alpha_{i-1}\cdots\alpha_1)\cdot W|
\end{align*}
unless all three lengths are equal.

Suppose $G<\AAG$, $S$ is a finite generating set for $G$, and $W$ is a $k$--tuple of conjugacy classes in $A_\Gamma$.
We say that $G$ has \emph{peak reduction} with respect to $W$ by elements of $S$ if every $\alpha\in G$ has a factorization by elements of $S$ that is peak-reduced with respect to $W$.
The peak-reduction theorem for a free group $F_n$ states that there is a finite generating set $\Omega$ for $\Aut F_n$ (called the \emph{Whitehead automorphisms}) such that $\Aut F_n$ has peak reduction with respect to any $k$--tuple of conjugacy classes $W$ in $F_n$ by elements of $\Omega$.
See Lyndon and Schupp~\cite{ls}, Chapter~I, Proposition~4.17 for a proof.
Whitehead's theorem is an important special case of this theorem.


The following definition of a Whitehead automorphism in $\AAG$ generalizes the definition of a Whitehead automorphism of a free group (see chapter I.4 of Lyndon and Schupp~\cite{ls}).
\pagebreak[3]
\begin{definition}\label{de:whauts}
A \emph{Whitehead automorphism} is an element $\alpha\in\AAG$ of one of the following two types:
\begin{description}
\item[Type~(1):] $\alpha$ restricted to $X\cup X^{-1}$ is a permutation of $X\cup X^{-1}$, or
\item[Type~(2):] there is an element $a\in X\cup X^{-1}$, called the \emph{multiplier} of $\alpha$, such that for each $x\in X$, the element $\alpha(x)$ is in $\{x,xa,a^{-1}x,a^{-1}xa\}$. 
\end{description}
Let $\Omega$ be the set of all Whitehead automorphisms of $A_\Gamma$.
\end{definition}
For our purposes, it is more natural to single out the following two subsets of $\Omega$:
\begin{definition}
A Whitehead automorphism $\alpha\in\Omega$ is \emph{long-range} if $\alpha$ is of type~(1) or if $\alpha$ is of type~(2) with multiplier $a\in X\cup X^{-1}$ and $\alpha$ fixes the elements of $X$ adjacent to $a$ in $\Gamma$.
Let $\Omega_\ell$ be the set of long-range elements of $\Omega$.

A Whitehead automorphism $\alpha\in\Omega$ is \emph{short-range} if $\alpha$ is of type~(2) with multiplier $a\in X\cup X^{-1}$ and $\alpha$ fixes the elements of $X$ not adjacent to $a$ in $\Gamma$.
Let $\Omega_s$ be the set of short-range elements of $\Omega$.
\end{definition}
It is easy to see that $\Omega$ is finite, and it is a consequence of the work of Laurence~\cite{la} (see Section~\ref{ss:Laurencesgenerators}) that $\OLS$ generates $\AAG$.
The following theorem is the main result of the current paper:
\begin{maintheorem}\label{mt:threeparts}
The finite generating set $\OLS$ for $\AAG$ has the following properties:
\begin{enumerate}
\item \label{it:complements} 
each $\alpha\in\AAG$ can be written as $\alpha=\beta\gamma$ for some $\beta\in \langle\Omega_s\rangle$ and some $\gamma\in\langle\Omega_\ell\rangle$;
\item
\label{it:srinj} 
the usual representation $\AAG\to \Aut H_1(A_\Gamma)$ to the automorphism group of the abelianization $H_1(A_\Gamma)$ of $A_\Gamma$ restricts to an embedding $\langle\Omega_s\rangle\into \Aut H_1(A_\Gamma)$; and
\item
\label{it:lrwhalg} the subgroup $\langle\Omega_\ell\rangle$ has peak reduction by elements of $\Omega_\ell$ with respect to any $k$--tuple $W$ of conjugacy classes in $A_\Gamma$.
\end{enumerate}
\end{maintheorem}
\newcommand\partref[1]{part~(\ref{#1}) of Theorem~\ref{mt:threeparts}}

The proof of Theorem~\ref{mt:threeparts} is effective: there is an algorithm that splits an automorphism into its $\langle \Omega_\ell\rangle$ and $\langle \Omega_s\rangle$ parts and an algorithm that peak-reduces an element of $\langle \Omega_\ell\rangle$.
Further, the theorem implies that we can analyze an element of $\langle \Omega_s\rangle$ by using row-reduction methods in $\Aut H_1(A_\Gamma)\cong \GL(n,\Z)$.
Note that if $A_\Gamma$ is a free group, $\Omega_\ell$ is Whitehead's generating set $\Omega$, our $\Omega_s$ contains only the identity, and Theorem~\ref{mt:threeparts} restricts to the peak-reduction theorem.

For most right-angled Artin groups, $\langle \Omega_\ell\rangle$ is a proper subgroup of $\AAG$, and \partref{it:lrwhalg} is seemingly weaker than the peak-reduction theorem for free groups.
In fact, we cannot hope for a straightforward generalization of peak reduction that applies to all of $\AAG$ for arbitrary $\Gamma$.
We show the following in Section~\ref{se:nowhalg}:
\begin{mainproposition}\label{mp:nowhalg}
There is a graph $\Gamma$ such that for every finite generating set $S$ of $\Aut A_\Gamma$, there is a conjugacy class $w$ in $A_\Gamma$ and an automorphism $\alpha\in\Aut A_\Gamma$ such that $\alpha$ cannot be peak-reduced with respect to $w$ by members of $S$.
\end{mainproposition}
In spite of this disappointing fact, there are still special cases where peak-reduction works.
As an example of such a special case, we prove the following:
\begin{mainproposition}\label{mp:specialwhalg}
Let $W=(w_1,\ldots,w_k)$ be a $k$--tuple of conjugacy classes such that for each $i$, $|w_i|=1$.
If $\alpha\in\AAG$ and $|\alpha\cdot W|=|W|$, then $\alpha$ can be peak-reduced with respect to $W$ by elements of $\Omega_\ell\cup\Omega_s$.
\end{mainproposition}

As an application of Theorem~\ref{mt:threeparts}, we give a presentation for $\AAG$.
In Section~\ref{ss:whgens}, we define a finite set $R$ of relations among the Whitehead automorphisms $\Omega$.
These relations tell us when one element of $\Omega$ is the inverse of another, when one is a product of two others, when two elements of $\Omega$ commute, when an element of $\Omega$ is the commutator of two other elements of $\Omega$, and how type (1) Whitehead automorphisms interact with type~(2) Whitehead automorphisms.
These relations are based on the relations McCool gives for the automorphism group of the free group in~\cite{mcpres}.
\begin{maintheorem}\label{mt:pres}
For any graph $\Gamma$, the group $\AAG$ is finitely presented.
Specifically, we have $\AAG=\langle \Omega | R\rangle$.
\end{maintheorem}
Although Bux--Charney--Vogtmann~\cite{bcv} have shown that $\AAG$ is finitely presented when $\Gamma$ is a tree, our result is more explicit and holds for arbitrary $\Gamma$.
The idea behind Theorem~\ref{mt:pres} is that we can use the methods of Theorem~\ref{mt:threeparts} to put any word in $\Omega$ representing the identity in $\AAG$ into a standard form.
We also use the fact that the inner automorphism group of a right-angled Artin group is also a right-angled Artin group.

\subsection{Acknowledgements}
Some of the results of this paper originally appeared in my Ph.D. thesis at the University of Chicago, and some of the research was done under the support of a graduate research fellowship from the National Science Foundation.
I am deeply grateful to Benson Farb, my thesis advisor, for many useful conversations and comments on earlier versions of this work.
I am grateful to Shmuel Weinberger for a conversation that led me to prove Proposition~\ref{mp:nowhalg}, and I am grateful to Adam Piggott for suggesting I use these methods to prove a theorem like Theorem~\ref{mt:pres}.
I am grateful to Karen Vogtmann for conversations about this project and I am grateful to Ruth Charney for conversations, and for helping me find an obscure reference.
I would also like to thank Hanna Bennett, Nathan Broaddus, Thomas Church, Jim Fowler and Benjamin Schmidt for comments on earlier versions of this paper.

\section{Generating sets for right-angled Artin groups}\label{se:background}
\subsection{Combinatorial group theory of $A_\Gamma$}
Let the set of letters $L$ be $X\cup X^{-1}$.
For $x\in L$, let $\pg{x}\in X$, the \emph{vertex of $x$}, be the unique element of $X\cap \{x,x^{-1}\}$. 
We will use $\adj{x}{y}$ as a shorthand for $\adj{\pg{x}}{\pg{y}}$ and we will use $\stl{x}$ and $\lkl{x}$ as notation for $\st(\pg{x})\cup\st(\pg{x})^{-1}$ and $\lk(\pg{x})\cup\lk(\pg{x})^{-1}$ respectively.

As usual, a word in $L$ represents an element in $A_\Gamma$.
A \emph{cyclic word} is a string of elements of $L$ indexed cyclically (or alternatively, an equivalence class of words under cyclic permutation of the indices).
Any two non-cyclic indexings of a cyclic word $w$ represent group elements that are conjugate to each other, so $w$ represents a well-defined conjugacy class. 
If $w$ is a cyclic word, we will use $[w]$ to denote the conjugacy class it represents.
If $w$ is a non-cyclic word, we will sometimes use $[w]$ to denote the cyclic word or conjugacy class it represents.

We will denote the length of a word or cyclic word $w$ by $|w|$.
The length of a group element or conjugacy class is the minimum length of any of its representative words or cyclic words, respectively.
A word or cyclic word $w$ on $L$ is \emph{graphically reduced} if it contains no subsegments of the form $aua^{-1}$, where $a\in L$ and $u$ is a word in $\lkl{a}$.
The \emph{support} $\supp w$ of a word or cyclic word $w$ is the subset of $X$ consisting of all generators that appear (or whose inverses appear) in $w$, and the support $\supp W$ of a $k$-tuple $W=(w_1,\ldots,w_k)$ of conjugacy classes is $\bigcup_{i=1}^k\supp w_i$.

According to Servatius (see~\cite{se}, Section I) any graphically reduced word can be transformed into any other graphically reduced representative of the same element by repeated application of commutation moves (replacing a subsegment $ab$ with $ba$ when $\adj{a}{b}$).
The same is true for cyclic words and conjugacy classes.
Therefore, we take the support $\supp w$ of a group element or conjugacy class to be the support of any graphically reduced representative.
The number of instances of a given generator in a group element or conjugacy class can be defined in the same way.

Servatius's centralizer theorem from~\cite{se}, Section III, finds all the centralizers of elements in $A_\Gamma$.
We restate a special case 
here:
\begin{theorem}[Special case of Servatius's centralizer theorem]\label{th:centralizer}
For $x\in X$, the centralizer of $x$ in $A_\Gamma$ is $\langle\stl{x}\rangle$.
\end{theorem}

\subsection{Laurence's generators for $\AAG$}\label{ss:Laurencesgenerators}
There is a reflexive, transitive, binary relation on $X$ called the \emph{domination relation}: say $x\geq y$ ($x$ dominates $y$) if $\lk(y)\subset \st(x)$.
Domination is clearly reflexive and transitive.
For $x, y\in L$, say $x\geq y$ if $\pg{x}\geq \pg{y}$.
Write $x\sim y$ when $x\geq y$ and $y\geq x$; the relation $\sim$ is called the \emph{domination equivalence} relation.
We will also consider the \emph{adjacent domination} relation, which holds for $x$ and $y$ if $\adj{x}{y}$ and $x\leq y$, and the \emph{non-adjacent domination} relation, which holds if $x\leq y$ and not $\adj{x}{y}$.
Each of these relations has a corresponding equivalence relation.
We say that $x$ strictly dominates $y$ if $x\geq y$ and $x\not\sim y$ (other authors have used ``strict domination" to refer to what we mean by ``non-adjacent domination").

The following classes of automorphisms were defined by Servatius in~\cite{se}, where he conjectured that they generate $\AAG$.
\begin{definition}
The Laurence--Servatius generators are the following four classes of automorphisms:

{\bf Dominated Transvections:}
For $x, y\in L$ with $x\geq y$ and $\pg{x}\neq \pg{y}$, the \emph{dominated transvection} (or simply \emph{transvection}) $\tau_{x,y}$ is the automorphism that sends
\[y \mapsto yx\] 
and fixes all generators not equal to $\pg{y}$.

{\bf Partial Conjugations:}
For $x\in L$ and $Y$ a union of connected components of $\Gamma - \st(\pg{x})$, the \emph{partial conjugation} $c_{x,Y}$ is the automorphism that sends
\[y \mapsto x^{-1}yx\quad\mbox{for $y\in Y$}\]
and fixes all generators not in $Y$. 

{\bf Inversions:}
For $x\in X$, the \emph{inversion} of $x$ is the automorphism that sends
\[x \mapsto x^{-1}\]
and fixes all other generators.

{\bf Graphic Automorphisms:}
For $\pi$ an automorphism of the graph $\Gamma$, the \emph{graphic automorphism} of $\pi$ is the automorphism that sends
\[x \mapsto \pi(x)\]
for each generator $x\in X$.
\end{definition}
It is a potential point of confusion that an ordinary conjugation automorphism is an example of a partial conjugation automorphism.

The following is Theorem~6.9 of Laurence~\cite{la}.
\begin{theorem}[Laurence]\label{th:lau}
The group $\AAG$ is generated by the finite set consisting of all dominated transvections, partial conjugations, inversions and graphic automorphisms of $A_\Gamma$.
\end{theorem}

%
%

\subsection{Whitehead automorphisms for right-angled Artin groups}\label{ss:whgens}
We start with some comments on the Whitehead automorphisms $\Omega$ defined the introduction.

There is a special notation for type~(2) Whitehead automorphisms.
Let $A\subset L$ and $a\in L$, such that $a\in A$ and $a^{-1}\notin A$.
If it exists, the symbol $(A,a)$ denotes the Whitehead automorphism satisfying
\[(A,a)(a)=a\]
and for $x\in X-\pg{a}$:
\[(A,a)(x)=\left\{\begin{array}{ll}
x & \mbox{ if $x\notin A$ and $x^{-1}\notin A$}\\
xa & \mbox{ if $x\in A$ and $x^{-1}\notin A$}\\
a^{-1}x & \mbox{ if $x\notin A$ and $x^{-1}\in A$}\\
a^{-1}xa & \mbox{ if $x\in A$ and $x^{-1}\in A$}
\end{array}\right.\]
Say that $(A,a)$ is \emph{well defined} if the formula given above defines an automorphism of $A_\Gamma$.
For $\alpha\in\Omega$ of type~(2), one can always find a multiplier $a\in L$ and a subset $A\subset L$ such that $\alpha=(A,a)$.
There is a little ambiguity in choosing such a representation that comes from the following fact: if $a,b\in L$ with $\adj{a}{b}$, then $(\{a,b,b^{-1}\},a)$ is the trivial automorphism.

Note that the set of type~(1) Whitehead automorphisms is the finite subgroup of $\AAG$ generated by the graphic automorphisms and inversions.

\begin{claim}
The set $\Omega$ of Whitehead automorphisms is a finite generating set for $\AAG$.
\end{claim}

\begin{proof}
The set $\Omega$ contains the Laurence-Servatius generators from Theorem~\ref{th:lau}.
There are only finitely many permutations of $L$ and finitely many subsets of $L$, so $\Omega$ is finite.
\end{proof}

\begin{lemma}\label{le:whdef}
For $A\subset L$ with $a\in A$ and $a^{-1}\notin A$, the automorphism $(A,a)$ is well defined if and only if both of the following hold:
\begin{enumerate}
\item The set $X\cap A\cap A^{-1}-\lk(\pg{a})$ is a union of connected components of $\Gamma - \st(\pg{a})$.
\item For each $x\in (A-A^{-1})$, we have $a\geq x$. 
\end{enumerate}
Alternatively, $(A,a)$ is well defined if and only if for each $x\in A-\stl{a}$ with $a\not\geq x$, $(A,a)$ acts on the entire component of $\pg{x}$ in $\Gamma-\st(\pg{a})$ by conjugation.
\end{lemma}
\begin{proof}
The alternate statement follows immediately from the first one.
For the ``only if" direction of the first statement, note that if both conditions hold, one can write $(A,a)$ as a product of the Laurence-Servatius generators.
For the other direction, assume either condition fails and $(A,a)$ defines an automorphism.
One can then find elements $x,y\in X$ such that $[x,y]=1$, but $[(A,a)(x),(A,a)(y)]\neq 1$ by Theorem~\ref{th:centralizer}.
This is a contradiction.
\end{proof}

\subsection{Relations among Whitehead automorphisms}
In this section we define the set of relations $R$ in Theorem~\ref{mt:pres}.
Note that we use function composition order and automorphisms act on the left.
With sets, we use the notation $A+B$ for $A\cup B$ when $A\cap B=\emptyset$.
Note the shorthands $A-a$ for $A-\{a\}$ and $A+a$ for $A+\{a\}$. 

Let $\Phi$ be the free group generated by the set $\Omega$.
We understand the relation ``$w_1=w_2$" to correspond to $w_1w_2^{-1}\in\Phi$.
Note that if $(A,a)\in\Omega$ with $B\subset \lk(\pg{a})$ and $(B\cup B^{-1})\cap A=\emptyset$, then $(A,a)$ and $(A+B+B^{-1},a)$ represent the same element of $\Omega$ and therefore the same element of $\Phi$.
This is why we do not list ``$(A,a)=(A+B+B^{-1},a)$" in the relations below.
 
\begin{definition}\label{de:identities}
The relations of type (R1) are
\begin{equation}\label{eq:R1}\tag{R1}
(A,a)^{-1}=(A-a+a^{-1},a^{-1})
\end{equation}
for $(A,a)\in \Omega$.

The relations of type (R2) are
\begin{equation}\label{eq:R2}\tag{R2}
(A,a)(B,a)=(A\cup B,a)
\end{equation}
for $(A,a)$ and $(B,a)\in\Omega$ with $A\cap B=\{a\}$.

The relations of type (R3) are
\begin{equation}\label{eq:R3}\tag{R3}
(B,b)(A,a)(B,b)^{-1}=(A,a)
\end{equation}
for $(A,a)$ and $(B,b)\in\Omega$ such that $a\notin B$, $b\notin A$, $a^{-1}\notin B$, $b^{-1}\notin A$, and at least one of (a) $A\cap B=\emptyset$ or (b) $b\in\lkl{a}$ holds.
We refer to this relation as (R3a) if condition (a) holds and (R3b) if condition (b) holds.

The relations of type (R4) are
\begin{equation}\label{eq:R4}\tag{R4}
(B,b)(A,a)(B,b)^{-1}=(A,a)(B-b+a,a)
\end{equation}
for $(A,a)$ and $(B,b)\in\Omega$ such that $a\notin B$, $b\notin A$, $a^{-1}\notin B$, $b^{-1}\in A$, and at least one of (a) $A\cap B=\emptyset$ or (b) $b\in\lkl{a}$ holds.
We refer to this relation as (R4a) if condition (a) holds and (R4b) if condition (b) holds.

The relations of type (R5) are
\begin{equation}\label{eq:R5}\tag{R5}
(A-a+a^{-1},b)(A,a)=(A-b+b^{-1},a)\sigma_{a,b}
\end{equation}
for $(A,a)\in \Omega$ and $b\in A$ with $b^{-1}\notin A$, $b\neq a$, and $b\sim a$, where $\sigma_{a,b}$ is the type~(1) Whitehead automorphism with $\sigma_{a,b}(a)=b^{-1}$, $\sigma_{a,b}(b)= a$ and which fixes the other generators.

The relations of type (R6) are
\begin{equation}\label{eq:R6}\tag{R6}
\sigma(A,a)\sigma^{-1}=(\sigma(A),\sigma(a))
\end{equation}
for $(A,a)\in\Omega$ of type~(2) and $\sigma\in\Omega$ of type~(1).

The relations of type (R7) are the entire multiplication table of the type~(1) Whitehead automorphisms, which form a finite subgroup of $\AAG$.

The relations of type (R8) are
\begin{equation}\label{eq:R8}\tag{R8}
(A,a)=(L-a^{-1},a)(L-A,a^{-1})
\end{equation}
for $(A,a)\in \Omega$. 

The relations of type (R9) are
\begin{equation}\label{eq:R9}\tag{R9}
(A,a)(L-b^{-1},b)(A,a)^{-1}=(L-b^{-1},b)
\end{equation}
for $(A,a)\in\Omega$ and $b\in L$ with $b,b^{-1}\notin A$.

The relations of type (R10) are
\begin{equation}\label{eq:R10}\tag{R10}
(A,a)(L-b^{-1},b)(A,a)^{-1}=(L-a^{-1},a)(L-b^{-1},b)
\end{equation}
for $(A,a)\in\Omega$ and $b\in L$ with $b\in A$ and $b^{-1}\notin A$.

Let $R$ be the set of elements of $\Phi$ corresponding to all relations of the forms (R1), (R2), (R3a), (R3b), (R4a), (R4b), (R5), (R6), (R7), (R9) and (R10).
\end{definition}
This is the same $R$ as in Theorem~\ref{mt:pres}, so we will show in Section~\ref{se:pres} that $\AAG=\langle \Omega | R\rangle$.
Note that $R$ is a finite set.

Relations (R1), (R2), (R3a), (R4a), and (R5)-(R10) appear for the automorphism group of the free group in McCool~\cite{mcpres} (McCool uses reverse composition order for his statements).
We have renamed McCool's (R3) as (R3a) and (R4) as (R4a).
Note that in Lyndon--Schupp~\cite{ls}, these relations also appear, but (R7) is unnamed, and (R8)-(R10) are relabeled as (R7)-(R9).

The relations (R3b) and (R4b) are new here.
Our statement of (R4a) varies from McCool's by an application of (R2); this allows us to give (R4b) as a relation of the same form.
Our statement of (R10) varies from McCool's by applications of (R1) and (R2), and our statement of (R9) varies from McCool's by an application of (R1).
This restatement should make it easier to see what relations (R9) and (R10) do.

\begin{remark} 
For $a\in L$, the automorphism $(L-a^{-1},a)$ is the inner automorphism given by conjugating by $a$.
Relation (R8) states that $(A,a)$ and $(L-A,a^{-1})$ represent the same element of $\OAG$.
Relations (R9) and (R10) are cases of the following familiar fact about groups: if for $g$ in a group $G$, $C_g$ denotes conjugation by $g$, then for any $\phi\in\Aut G$, we have $\phi C_g\phi^{-1}=C_{\phi(g)}$.

In the case of the free group, Relations~(R8), (R9), and (R10) follow from relations (R1)-(R7).
However, in the case of a general right-angled Artin group, this is only true of Relation~(R8), which follows immediately from Relations~(R1) and (R2).
This is why we leave relations of type (R8) out of $R$.
However, we leave Relation~(R8) in the list of relations for convenience and to keep with McCool's numbering system.
Relations~(R9) and~(R10) follow from the other relations only if conjugation automorphisms can be factored into products of dominated transvections; this is not possible for general $A_\Gamma$. 
\end{remark}

\begin{proposition}\label{pr:identities}
For each relation $w$ in any of the classes of relations (R1)-(R10), all the symbols appearing in $w$ denote well-defined Whitehead automorphisms.
Furthermore, these relations are true identities in $\AAG$.
\end{proposition}

\begin{proof}
If $w$ is a relation of type~(R3) or (R7), then it is vacuously true that all the terms appearing in $w$ are well defined (since the instances of these relations are indexed over well-defined terms).

If $w$ is a relation of type~(R1), (R2), (R5) or (R6), then type~(2) Whitehead automorphisms in $w$ are clearly well defined by Lemma~\ref{le:whdef}.

If $w$ is a relation of type~(R4), then since $b\notin A$ but $b^{-1}\in A$, we know $a\geq b$ (by Lemma~\ref{le:whdef}).
Since $a\geq b$, every component of $\Gamma-\st(\pg{b})$ is a union of components of $\Gamma-\st(\pg{a})$ and elements of $\st(a)$.
Then by Lemma~\ref{le:whdef}, $(B-b+a,a)$ is well defined.

If $w$ is a relation of type~(R5), then we have $a\sim b$ with $a\neq b$, which implies that there is an automorphism $\pi$ of $\Gamma$ switching $\pg{a}$ and $\pg{b}$ but fixing the other vertices.
Then $\sigma_{a,b}$ is the composition of the automorphism of $A_\Gamma$ induced from $\pi$ with the inversion of $b$.

If $w$ is a relation of type~(R8), then $(L-A,a^{-1})$ is well defined by Lemma~\ref{le:whdef}.
Since for any $c\in L$, $(L-c^{-1},c)$ represents conjugation by $c$, we know the terms in $w$ are well defined if $w$ is a relation of type~(R8), (R9), or (R10).

Each identity can then be verified by computing actions on $X$.
\end{proof}

\begin{remark}
At this point it is easy to see that $\Omega=\Omega^{-1}$.
This is because of Equation~(\ref{eq:R1}) and the fact that the set of type~(1) Whitehead automorphisms is closed under taking inverses.
\end{remark}




\section{The structure of $\AAG$}
\subsection{Sorting automorphisms by their scope}\label{ss:sorting}
Using the special notation for type (2) Whitehead automorphisms, we can restate the definitions of $\Omega_s$ and $\Omega_\ell$ more succinctly.
A Whitehead automorphism $\alpha$ is in $\Omega_\ell$ if it is of type~(1), or if $\alpha$ is of type (2) and we can write $\alpha=(A,a)$ for some $A$ with $A\cap\lkl{a}=\emptyset$.
A Whitehead automorphism $\alpha$ is in $\Omega_s$ if $\alpha=(A,a)$ is of type (2) and $A\subset\stl{a}$.

Whenever we declare an element $(A,a)\in\Omega_\ell$, we will assume that $A\cap\lkl{a}=\emptyset$.
This is necessary since if $x\in\lkl{a}$, then $(A\cup\{x,x^{-1}\},a)=(A,a)$.

The goal of this subsection is to prove \partref{it:complements}.
We proceed by describing a series of identities that allow us to rewrite a product of a long-range and a short-range automorphism.
We will then show that by a finite number of applications of these identities, we can express any automorphism as a product of a single element of $\langle \Omega_s\rangle$ and a single element of $\langle \Omega_\ell\rangle$.

\begin{lemma}\label{le:l&s}
Every Whitehead automorphism $\alpha\in\Omega$ has a unique decomposition as a product
$\alpha=\alpha_s\alpha_\ell$,
where $\alpha_s\in\Omega_s$ and $\alpha_\ell\in\Omega_\ell$.
Furthermore, $\AAG$ is generated by $\OLS$. 
\end{lemma}
\begin{proof}
If $\alpha$ is of type~(1), then $\alpha_s=1$ and $\alpha_\ell=\alpha$.
So assume $\alpha=(A,a)$.
Set $A_1= A\cap\stl{a}$ and set $A_2=A-\lkl{a}$.
By Lemma~\ref{le:whdef}, both $(A_1,a)$ and $(A_2,a)$ are well defined.
So set $\alpha_s=(A_1,a)$ and set $\alpha_\ell=(A_2,a)$.
By Equation~(\ref{eq:R2}), we have $\alpha=\alpha_s\alpha_\ell$.
This decomposition is unique since $\Omega_s\cap\Omega_\ell=\{1\}$.

Of course, this means that $\Omega\subset\langle\OLS\rangle$.
Then since $\Omega$ generates $\AAG$, the set $\OLS$ also generates $\AAG$.
\end{proof}

We call $\alpha_s$ the \emph{short-range part} of $\alpha$ and $\alpha_\ell$ the \emph{long-range part} of $\alpha$.
Let $\ell\co\Omega\to\Omega_\ell$ be given by $\ell(\alpha)=\alpha_\ell$
and let $s\co\Omega\to\Omega_s$ be given by $s(\alpha)=\alpha_s$.

\begin{definition}\label{de:sort}
Suppose $\alpha\in\Omega_\ell$ and $\beta=(B,b)\in\Omega_s$.
Of course, we may assume that for $x\in\lkl{b}$, not both $x$ and $x^{-1}$ are in $B$.
Let the \emph{sorting substitution} of $\alpha\beta$ be the word in $\Omega$ defined as follows.

If $\alpha$ is given by a permutation of $L$, 
then $(\alpha(B),\alpha(b))\in \Omega_s$, and 
the substitution is:
\begin{equation}\label{eq:permsubs}
\alpha\beta\mapsto (\alpha(B),\alpha(b))\alpha
\end{equation}

Now suppose $\alpha=(A,a)$.
If $\pg{a}=\pg{b}$, then the substitution is given by:
\begin{equation}\label{eq:vertsubs}
\alpha\beta\mapsto \beta\alpha
\end{equation}

If $a\in\lkl{b}$, then we know $b,b^{-1}\notin A$.
As we assumed earlier, not both $a\in B$ and $a^{-1}\in B$.
The substitution is given by:
\begin{equation}\label{eq:adjsubs}
\alpha\beta\mapsto\left\{\begin{array}{ll}
\beta\alpha & \mbox{if $a\notin B$, $a^{-1}\notin B$}\\
\beta s((A-a+b,b))\ell((A-a+b,b))\alpha & \mbox{if $a\notin B$, $a^{-1}\in B$}\\
\beta s((A-a+b^{-1},b^{-1}))\ell((A-a+b^{-1},b^{-1}))\alpha & \mbox{if $a\in B$, $a^{-1}\notin B$}
\end{array}\right.
\end{equation}
If $a\notin\stl{b}$,
the substitution is given by:
\begin{equation}\label{eq:farsubs}
\alpha\beta\mapsto \beta\alpha
\end{equation}
\end{definition}

\begin{sublemma}\label{sl:domadjdom}
Suppose $b\in L$, $c\in\lkl{b}$ and $b\geq c$.
If $a\in L$ and $a\geq b$, then $b\in\stl{a}$.
\end{sublemma}

\begin{proof}
By transitivity, $a\geq c$.
Then by the definition of domination, $b\in\stl{a}$.
\end{proof}

\begin{lemma}\label{le:validsubs}
All of the elements substituted for $\alpha\beta$ in Definition~\ref{de:sort} 
are equal to $\alpha\beta$ in $\AAG$.
\end{lemma}
\begin{proof}
Substitution~(\ref{eq:permsubs}) is valid by Equation~(\ref{eq:R6}).
Substitution~(\ref{eq:vertsubs}) is valid by Equation~(\ref{eq:R2}) (if $a=b$, then we know $A\cap B=\{a\}$ and therefore $\alpha\beta=(A\cup B,a)=\beta\alpha$) and Equation~(\ref{eq:R1}) (if $a^{-1}=b$, then $\alpha^{-1}=(L-A-\lkl{a},a^{-1})$ and Equation~(\ref{eq:R2}) implies $\alpha^{-1}$ and $\beta$ commute).
The first substitution in Equation~(\ref{eq:adjsubs}) is valid by Equation~(\ref{eq:R3}b).
The second substitution in Equation~(\ref{eq:adjsubs}) is valid by Equation~(\ref{eq:R4}b) (and Equation~(\ref{eq:R2}) to split $(A-a+b,b)$ into long-range and short-range parts).
To get the third substitution in Equation~(\ref{eq:adjsubs}), note that since $\beta\in\Omega_s$, we have $\beta=(\stl{b}-B,b^{-1})$.
Then the third substitution is simply the second substitution, after representing $\beta$ in a different way.

Now suppose $a\notin\stl{b}$.
By Sublemma~\ref{sl:domadjdom} (assuming $\beta$ is nontrivial), we know that $a\not\geq b$.
Then if $b\notin A$, each element of the component of $\pg{b}$ in $\Gamma-\st(\pg{a})$ is fixed by $\alpha$.
Then $A\cap B=\emptyset$ since $\beta\in\Omega_s$.
So if $b\notin A$, then Equation~(\ref{eq:farsubs}) is valid by Equation~(\ref{eq:R3}).
If $b\in A$, we apply Equation~(\ref{eq:R8}) to replace $(A,a)$ by $(L-a^{-1},a)(L-A,a^{-1})$.
Then $b\notin L-A$, so $(L-A,a^{-1})$ commutes with $\beta$ by Equation~(\ref{eq:R3}), and $(L-a^{-1},a)$ commutes with $\beta$ by Equation~(\ref{eq:R9}) ($a,a^{-1}\notin B$ since $\beta\in\Omega_s$).
After commuting both automorphisms past $\beta$, we recombine them by Equation~(\ref{eq:R8}).
\end{proof}

\begin{remark}
Equation~(\ref{eq:adjsubs}) indicates that there are many examples of graphs $\Gamma$ such that neither $\langle\Omega_\ell\rangle$ nor $\langle\Omega_s\rangle$ is a normal subgroup of $\AAG$.
\end{remark}

\begin{lemma}\label{le:1sr}
Suppose we have $(A,a)\in \Omega_\ell$.
Suppose $w$ is a product (in any order) of long-range automorphisms of the form $\ell((A-a+x,x))$ for $x\in\stl{a}$ with $x\geq a$,
 together with short-range automorphisms with multiplier $b^{\pm1}$.
Then we can apply sorting substitutions to rewrite $w$ as a word $w'$ satisfying the same hypotheses as $w$ (for the same $(A,a)$), and such that all the short-range automorphisms in $w'$ appear to the left of any long-range automorphisms in $w'$.
\end{lemma}

\begin{proof}
We argue by induction on the number $k$ of long-range automorphisms in $w$ with multipliers other than $b^{\pm1}$.
It is clear that by applying Substitution~(\ref{eq:vertsubs}), we can move a short-range automorphism appearing in the word to the left across a long-range automorphism with multiplier $b^{\pm1}$.
In the base case $k=0$, we only need to move short-range automorphisms across long-range automorphisms with multiplier $b^{\pm1}$, so we are done.
Now suppose $k>0$.
We break up $w$ as $w_1w_2$, where there is only one long-range automorphism with multiplier not equal to $b^{\pm1}$ in $w_2$, say $\ell((A-a+x,x))$.
If $x\notin\stl{b}$, then by applying Substitution~(\ref{eq:farsubs}) (and possibly Substitution~(\ref{eq:vertsubs})), we can move all the short-range automorphisms in $w_2$ to the left across $\ell((A-a+x,x))$.
If $x\in\lkl{b}$, then by applying Substitution~(\ref{eq:adjsubs}), we can move any short-range automorphism to the left across $\ell((A-a+x,x))$.
In doing so, we may introduce a new short-range automorphism with multiplier $b^{\pm1}$ to the left of $\ell((A-a+x,x))$, as well as a new long-range automorphism $\ell((A-a+y,y))$ where $y=b^{\pm1}$.
It is then clear that by applying Substitutions~(\ref{eq:adjsubs}) and~(\ref{eq:vertsubs}), we can move all the short-range automorphisms to the left across $\ell((A-a+x,x))$.
In either case, we can rewrite $w_2$ as a word $w_3v$ satisfying the hypotheses of the lemma, where $w_3$ contains no long-range elements with multipliers other than $b^{\pm1}$ and $v$ contains no short-range elements.
Then by induction, $w_1w_3$ can be rewritten as a word $u$ satisfying the conclusions of the lemma, and we have rewritten $w$ as a word $uv$ satisfying the conclusions of the lemma.
\end{proof}

\begin{lemma}\label{le:1lr}
Suppose $\alpha=(A,a)\in\Omega_\ell$ and $\beta_1,\ldots,\beta_k\in \Omega_s$.
Then we can apply finitely many sorting substitutions to the word
$w_0=\alpha\beta_k\cdots\beta_1$
 to get a word where all of the long-range elements are of the form $\ell((A-a+x,x))$ for various $x\in \stl{a}$ with $x\geq a$, and all the short-range elements are to the left of any long-range elements.
\end{lemma}

\begin{proof}
We prove the lemma by induction on $k$.
If $k=0$, it is true.
Now assume the lemma holds for $\alpha\beta_k\cdots\beta_2$.
Then we rewrite $w_0$ as
$u_1\delta_m\cdots\delta_1\beta_1$,
where $u_1$ is a word in $\Omega_s$ and $\delta_i=\ell((A-a+x_i,x_i))$ for some $x_1,\ldots,x_m\in \stl{a}$ with $x_i\geq a$.
We apply Lemma~\ref{le:1sr} to the subsequence $\delta_m\cdots\delta_1\beta_1$ and rewrite it as $u_2v$, where $u_2$ is a word in $\Omega_s$ and $v$ is a product of automorphisms of the form $\ell((A-a+x,x))$ for various $x\geq a$.
Then we have rewritten $w_0$ as $u_1u_2v$, which is in the desired form.
\end{proof}

\begin{proof}[Proof of \partref{it:complements}]
We induct on the length of the word $w$ in $\OLS$ that we wish to rewrite.
If $|w|\leq1$, we are done.
Now assume the theorem is true for words of length $|w|-1$.
Let $w'$ be a word in $\OLS$ and let $\alpha\in\OLS$ such that $w=\alpha w'$ is a reduced factorization.
Then the theorem applies to $w'$, so $w'=w_sw_\ell$ where $w_s$ is a word in $\Omega_s$ and $w_\ell$ is a word in $\Omega_\ell$.
If $\alpha\in\Omega_s$, we are done, so assume $\alpha\in\Omega_\ell$.
If $\alpha$ is induced by a permutation of $L$, we can move it across $w_s$ by $|w_s|$ applications of Substitution~(\ref{eq:permsubs}) and we are done.
Otherwise, $\alpha=(A,a)$, and we apply Lemma~\ref{le:1lr} to rewrite $\alpha w_s$ as $w'_sw'_\ell$, with $w'_s$ a word in $\Omega_s$ and $w'_\ell$ a word in $\Omega_\ell$.
Then $w=w'_sw'_\ell w_\ell$ and we are done.
\end{proof}


\subsection{The homology representation and short-range automorphisms}

Let $H_\Gamma$ denote the abelianization $H_1(A_\Gamma)\cong \Z^n$ of $A_\Gamma$.
Since the commutator subgroup of $A_\Gamma$ is a characteristic subgroup of $A_\Gamma$, 
every automorphism of $A_\Gamma$ induces an automorphism of $H_\Gamma$.
This defines a map $\AAG\to\Aut H_\Gamma\cong \GL(n,\Z)$, which we call the \emph{homology representation}.
For $\alpha\in\AAG$, we denote its image under the homology representation by $\alpha_*\in\Aut H_\Gamma$.
In this section, we prove \partref{it:srinj}, and we examine the structure of $\langle \Omega_s\rangle$.

\begin{lemma}\label{le:srsupp}
Let $\beta\in\langle\Omega_s\rangle$ and let $c\in X$.
Then $\supp \beta(c)$ is a clique contained in $\st(c)$.
\end{lemma}

\begin{proof}
We proceed by induction on the length of $\beta$ as a product of members of $\Omega_s$.
If $|\beta|=0$, the lemma holds.
Now suppose that $w$ is a word such that $\supp w$ is a clique contained in $\st(c)$ and that $(B,b)\in\Omega_s$.
Of course, we may assume that $B\cap B^{-1}=\emptyset$.
If $(B-\{b\})\cap (\supp w)^{\pm1}=\emptyset$, then $\beta(w)=w$.
So suppose we have $a\in (B-\{b\})\cap (\supp w)^{\pm1}$.
Then $a\in\stl{c}$ and by Lemma~\ref{le:whdef}, $b\geq a$.
Since $\beta\in\Omega_s$, $b$ is adjacent to $a$, and since $\supp w$ is a clique, it follows from the definition of domination that $b$ is adjacent to every other member of $\supp w$.
Since $\supp (B,b)(w)\subset \{\pg{b}\}\cup\supp w$, this proves the lemma.
\end{proof}

\begin{proof}[Proof of \partref{it:srinj}]
Suppose $\beta\in\langle\Omega_s\rangle$ and $\beta_*\in \Aut H_\Gamma$ is trivial.
Then for any $a\in X$, it follows from Lemma~\ref{le:srsupp} that we can commute the elements of $\supp\beta(a)$ past each other.
But since $\beta_*$ is trivial, the sum exponent of $a$ in $\beta(a)$ is $1$, and the sum exponent of any other $x$ in $\beta(a)$ is $0$.
So $\beta(a)=a$ for any $a$, and
$\beta$ is trivial.
\end{proof}

Now we examine the structure of $\langle \Omega_s\rangle$.
The argument below tells us the structure of the images of many subgroups of $\AAG$ under the homology representation, so we phrase it in greater generality.
In particular, we prove an intermediate result that is quoted in the sequel to the current paper~\cite{ssraag}.

Let $\leq'$ be a transitive, reflexive relation on $X$, such that $a\leq'b$ implies $a\leq b$ for $a,b\in X$ (for example, the adjacent domination relation).
Write $a\sim' b$ when $a\leq' b$ and $b\leq' a$; then $\sim'$ is an equivalence relation on $X$.
Let $G<\Aut H_\Gamma$ be generated by $\{(\tau_{a,b})_*|\text{$a\geq'b$}\}$.
Let $C_1\cup \cdots\cup C_m = X$ be the $\sim'$--classes of $X$.
Let $N=\langle\{(\tau_{a,b})_*|\text{$a,b\in X$, $a\geq' b$ and $a\not\sim' b$}\}\rangle$
and for each $i=1,\ldots,m$, let $G_i=\langle\{(\tau_{a,b})_*|\text{$a,b\in C_i$}\}\rangle$.
\begin{proposition} \label{pr:homrepimstruct}
The group $N$ is nilpotent, each $G_i\cong \SL(|C_i|,\Z)$, and the inclusion maps of $N$ and the $G_i$ into $G$ give the decomposition: 
\begin{equation}\label{eq:transstruct}G\cong \left(G_1\times \cdots\times G_m\right)\ltimes N\end{equation}
\end{proposition}

\begin{proof}
We can pick an indexing $X=\{x_1,\ldots,x_n\}$ such that if $x_i\geq' x_j$ with $x_i\not\sim' x_j$, then $j> i$.
Taking the image of $X$ in $H_\Gamma$ to be an ordered basis under this indexing, the homology representation takes the transvection $\tau_{x_i,x_j}$ to the elementary matrix $E_{i,j}$.
Then the group $N$ is then a group of upper-triangular unipotent matrices and is therefore nilpotent.

We make a further demand on our indexing of $X$: if $i<j<k$ and $x_i\sim' x_k$, then $x_j\sim' x_i$.
Under such an indexing, the elements of $C_i$ are an unbroken string of elements of $X$, so say $C_i=\{x_{r_i},\ldots,x_{s_i}\}$.
So $G_i$ is generated by the elementary matrices $E_{j,k}$ with $r_i\leq j,k\leq s_i$; in particular, it is an embedded copy of $\SL(|C_i|,\Z)$.

The generators of $G_i$ and $G_j$ commute for $i\neq j$ by Equation~(\ref{eq:R3}), so the subgroup generated by the $\{G_i\}_i$ is a direct product.

It is obvious that $N$ and $G_1,\ldots,G_m$ generate $G$.
Now suppose that $(\tau_{a,b})_*$ is a generator of $N$ and $(\tau_{c,d})_*$ is a generator of one of the $G_i$.
Then $(\tau_{c,d})_*^{-1}(\tau_{a,b})_*(\tau_{c,d})_*$ is an element of $N$;
if $d\neq b$, this follows from Equations~(\ref{eq:R3}) and~(\ref{eq:R4}),
and if $d=b$, then this element is $(\tau_{a,b})_*$ since we are working in $H_\Gamma$.
Since $N$ is normal, we get the decomposition of $G$ in Equation~(\ref{eq:transstruct}).
\end{proof}


\begin{proposition}\label{pr:linpres}
The group $G$ has a presentation in which the generators $S_G$ are the row operations $E_{a,b}=(\tau_{a,b})_*$ for $a, b\in X$ with $a\geq'b$, and with the relations $R_G$ being all relations among the $S_G$ of the following forms:
\begin{enumerate}
\item\label{it:r1}
$[E_{a,b},E_{c,d}]=1$ if $b\neq c$ and $a\neq d$,
\item\label{it:r2}
$[E_{a,b},E_{b,d}]E_{a,d}^{-1}=1$ if $a\neq d$,
\item\label{it:r3}
$(E_{a,b}E_{b,a}^{-1}E_{a,b})^4=1$, if $a\sim' b$ and $a\neq b$,
\item\label{it:r4}
$(E_{a,b}E_{b,a}^{-1}E_{a,b})^2(E_{a,b}E_{b,a}^{-1}E_{a,b}E_{b,a})^{-3}=1$, if $a,b\in C_i$, $a\neq b$, for some $i$ with $|C_i|=2$.
\end{enumerate}
\end{proposition}

\begin{proof}
By Proposition~\ref{pr:homrepimstruct}, each $G_i\cong\SL(|C_i|,\Z)$ and $N$ is a nilpotent group.
For each $i$, it follows from classical presentations for $\SL(n,\Z)$ that $G_i$ has a presentation with generators $S_G\cap G_i$ and whose relations are those elements of $R_G$ only involving the generators in $S_G\cap G_i$ (see Corollary 10.3 of Milnor~\cite{mil} for $n\geq 3$ and Example~4.2(c) of Section~I.4 of Serre~\cite{sertree} for $n=2$).
Since $N$ is a unipotent matrix group and $S_G\cap N$ is a generating set for $N$ that is closed under taking commutators, we know that $N$ has a presentation with generators $S_G\cap N$ and whose relations are those elements of $R_G$ only involving generators in $S_G\cap N$.

Relation~(\ref{it:r1}) implies that the group generated by $S_G\cap (G_1\cup\cdots\cup G_m)$ subject to these relations is isomorphic to the product $G_1\times\cdots \times G_m$, and Relation~(\ref{it:r1}) and Relation~(\ref{it:r2}) encode the semi-direct product action of $G_1\times\cdots\times G_m$ on $N$, so that the group $\langle S_G|R_G\rangle\cong G$.
\end{proof}

\begin{corollary}\label{co:tvgenstruct}
Suppose $\tilde G$ is a subgroup of $\AAG$ generated by a set of dominated transvections.
Then there is a relation $\leq'$ such that the image of $\tilde G$ under the homology representation is the group $G$ generated by $\{(\tau_{a,b})_*|\text{$a\geq'b$}\}$ and the conclusions of Proposition~\ref{pr:homrepimstruct} and Proposition~\ref{pr:linpres} hold for $G$.
In particular, the group $\langle \Omega_s\rangle$ has a decomposition of the form in Equation~(\ref{eq:transstruct}) and a presentation of the form given in Proposition~\ref{pr:linpres}.
\end{corollary}

\begin{proof}
Let $S$ be a set of dominated transvections such that $\tilde G=\langle S\rangle$.
Let $\leq'$ be the reflexive relation that is the transitive closure of $\leq''$, where $a\geq''b$ whenever $\tau_{a,b}\in S$.
Then $\tilde G$ is generated by $\{\tau_{a,b}|\text{$a\geq'b$}\}$, and its image under the homology representation is the group $G$ generated by $\{(\tau_{a,b})_*|\text{$a\geq'b$}\}$.
Then we can apply Proposition~\ref{pr:homrepimstruct} and Proposition~\ref{pr:linpres}.
If $S=\Omega_s$, then by \partref{it:srinj}, the homology representation restricted to $\langle \Omega_s\rangle=\tilde G$ maps isomorphically to $G$.
So the conclusions of Proposition~\ref{pr:homrepimstruct} and Proposition~\ref{pr:linpres} apply to $\langle \Omega_s\rangle$, where $\leq'$ is adjacent domination.
\end{proof}

\subsection{Peak-reducing products of long-range automorphisms}
The goal of this subsection is to prove \partref{it:lrwhalg}.
Our proof 
is similar to the proof of the peak reduction theorem for free groups given in Higgins--Lyndon~\cite{hl}. 

Let $k\geq 1$.
For a $k$--tuple $W=(w_1,\ldots,w_k)$ of cyclic words, we denote the $k$--tuple of conjugacy classes by $[W]=([w_1],\ldots,[w_k])$.
We proceed by putting the definition of peak reduction from the introduction in context.

\begin{definition}
Suppose $\alpha,\beta \in\Omega$ and $[W]$ is a $k$--tuple of conjugacy classes in $A_\Gamma$.
Then $\beta\alpha$, the word of length $2$, is called a \emph{peak} with respect to $[W]$ if 
\[|\alpha\cdot [W]|\geq |[W]|\]
\[|\alpha\cdot [W]|\geq |\beta\alpha\cdot [W]|\]
and at least one of these inequalities is strict.

Suppose $\gamma\in\AAG$ and we have a factorization $\gamma=\alpha_k\cdots\alpha_1$
with $\alpha_1,\ldots,\alpha_k$ in $\Omega$.
We say $\alpha_i$ is a \emph{peak} of this factorization, with respect to $[W]$, if
$1<i<k$ and $\alpha_{i+1}\alpha_i$ is a peak with respect to $(\alpha_{i-1}\cdots\alpha_1)\cdot [W]$.
The $\emph{height}$ of a peak $\alpha_i$ is simply $|(\alpha_i\cdots\alpha_1)\cdot [W]|$.
\end{definition}

Then the factorization $\gamma=\alpha_k\cdots\alpha_1$ is peak-reduced with respect to $[W]$ (as defined in the introduction) if and only if it has no peaks with respect to $[W]$.

It is important to note that for general right-angled Artin groups, the automorphism group $\AAG$ does not act on the set of graphically reduced words; rather it only acts on the set of group elements.
This means we need to take care to distinguish words from the elements they represent.
These measures were unnecessary in the original proof for free groups, since for a free group the set of reduced words is the set of group elements.

\begin{definition}\label{de:nicerep}
If $(A,a)\in\Omega_\ell$ and $w$ is a graphically reduced cyclic word, define the \emph{obvious representative} of $(A,a)([w])$ based on $w$ to be the cyclic word $w'$ gotten from $w$ by the following replacements:
\begin{itemize}
\item for every subsegment of $w$ of the form $buc^{-1}$ or $cub^{-1}$, with $u$ any word in $\lkl{a}$, $b\in A-a$ and $c\in L-A-\lkl{a}$, replace this subsegment with $bauc^{-1}$ or $cua^{-1}b^{-1}$ respectively in $w'$, and
\item for every subsegment of $w$ of the form $bua^{-1}$ or $aub^{-1}$, with $u$ any word in $\lkl{a}$ and $b\in A-a$, replace this subsegment with $bu$ or $ub^{-1}$ respectively in $w'$.
\end{itemize} 

The \emph{obvious representative} of $(A,a)\cdot [W]$ based on $W$ is the $k$--tuple $(w'_1,\ldots,w'_k)$, where each $w'_i$ is the obvious representative of $(A,a)([w_i])$ based on $w_i$.
\end{definition}

\begin{claim}The obvious representative $w'$ of $(A,a)([w])$ based on $w$ is a graphically reduced representative of $(A,a)([w])$.
\end{claim}
\begin{proof}
First we show that $w'$ is graphically reduced.
Note that we have only added or removed instances of $a^{\pm1}$.
Since $w$ is graphically reduced, $w'$ can only fail to be graphically reduced on subsegments where we added or removed instances of $a^{\pm1}$.
Those replacements that introduce an instance of $a^{\pm1}$ introduce it in a way where it cannot cancel ($a$ does not commute with $b\in A-a$ or with $c\in L-A-\lkl{a}$).
Suppose a replacement that removes an instance of $a^{\pm1}$ results in $w'$ not being graphically reduced.
Then we have a subsegment $dvd^{-1}$ of $w$ being replaced by $dv'd^{-1}$ in $w'$, where $d\in L-\stl{a}$, $v'$ is a word in $\lkl{d}$, and $v$ is $v'$ with some instances $a$ or $a^{-1}$ inserted. 
Then $dvd^{-1}$ contains an instance of some $b\in (A-a)^{\pm1}$.
If $a\not\geq b$, then $\pg{b}$ is in a component of $\Gamma-\st(\pg{a})$ that is conjugated by $(A,a)$.
It follows that for each $x$ in $\supp dv'd^{-1}$, either $x\in\stl{a}$ or $\pg{x}$ is in the same component of $\Gamma-\st(\pg{a})$ as $\pg{b}$.
Then $v=v'$, a contradiction.
So suppose $a\geq b$.
If $\pg{b}\neq\pg{d}$,
then $d$ commutes with $a$, a contradiction.
If $\pg{b}=\pg{d}$, then $\supp v\subset\st(\pg{a})$, and our substitutions never would have removed instances of $a^{\pm1}$ from $v$.
So $w'$ is graphically reduced.

Observe that $w'$ represents $(A,a)([w])$: because $(A,a)$ is long-range, it has no effect on the subsegments of $w$ that are words in $\lkl{a}$, and the substitutions in the definition of $w'$ capture all those changes that $(A,a)$ makes to $w$ that do not cancel each other out. 
\end{proof}

\begin{definition}\label{de:adjco}
Let $w$ be a graphically reduced cyclic word and let $a\in L$.
Then for $b,c\in L-\lkl{a}$, we define the \emph{adjacency counter} of $w$ relative to $a$, written as
$\aco{b}{c}{w,a}$,
to be the number of subsegments of $w$ of the form $(buc^{-1})^{\pm1}$, where $u$ is any (possibly empty) word in $\lkl{a}$.

For a $k$--tuple of graphically reduced cyclic words $W=(w_1,\ldots,w_k)$, define the adjacency counter of $W$ relative to $a$ as:
\[\aco{b}{c}{W,a}=\sum_{i=1}^k\aco{b}{c}{w_i,a}\]
For $B,C\subset L$, we define:
\[\aco{B}{C}{W,a}=\sum_{b\in (B-\lkl{a})}\sum_{c\in (C-\lkl{a})}\aco{b}{c}{W,a}\]
For $\alpha\in\AAG$, we define:
\[D_{[W]}(\alpha)=|\alpha\cdot [W]|-|[W]|\]
When $W$ is clear, we leave it out, writing $\aco{B}{C}{a}$ and $D(\alpha)$.
\end{definition}

With $W$ and $a$ as above, note that for any $B,C\subset L$, the number $\aco{B}{C}{a}\geq 0$.
Further, we have
$\aco{B}{C}{a}=\aco{C}{B}{a}$.
If $D\subset L$ with $D\cap C=\emptyset$, then we have:
\[\aco{B}{C+D}{a}=\aco{B}{C}{a}+\aco{B}{D}{a}\]
Also note that $\aco{a}{a}{a}=0$ (since each $w_i$ is graphically reduced).

\begin{lemma}\label{le:obvadj}
If $W$ is a $k$--tuple of graphically reduced cyclic words, $(A,a)\in\Omega_\ell$, and $W'$ is the obvious representative of $(A,a)\cdot [W]$, then:
\[D_{[W]}((A,a))=|W'|-|W|=\aco{A-a}{L-A}{W,a}-\aco{a}{A-a}{W,a}\]
\end{lemma}

\begin{proof}
This is immediate from counting the letters removed and added in the definition of $W'$.
\end{proof}

Note that if $W'$ and $W$ are different $k$-tuples of graphically reduced cyclic words representing the same $k$-tuple of conjugacy classes, we may have different adjacency counters with respect to $W$ and $W'$.
However, the function $D_{[W]}$ depends only on $[W]$.

\begin{lemma}\label{le:dico}
Let $W$ be a $k$--tuple of graphically reduced cyclic words.
If $(A,a)\in\Omega_\ell$, then
\[D_{[W]}((A,a))=\aco{A}{L-A}{W,a}-\aco{a}{L}{W,a}\]
\end{lemma}

\begin{proof}
From Lemma~\ref{le:obvadj}:
\[
\begin{split}
D((A,a))&=\aco{A-a}{L-A}{a}-\aco{a}{A-a}{a}\\
&=\aco{A}{L-A}{a}
-(\aco{a}{L-A}{a}+\aco{a}{A-a}{a}+\aco{a}{a}{a})\\
&
=\aco{A}{L-A}{a}-\aco{a}{L}{a}
\end{split}
\] 
\end{proof}

The following lemma is the machine that makes peak reduction possible.
This is an extension of a parallel lemma for free groups that appears in Higgins--Lyndon~\cite{hl}.

\begin{lemma}\label{le:machine}
Suppose $\alpha,\beta\in \Omega_\ell$ and $[W]$ is a $k$--tuple of conjugacy classes of $A_\Gamma$.
If $\beta\alpha^{-1}$ forms a peak with respect to $[W]$, there exist $\delta_1,\ldots,\delta_k\in \Omega_\ell$ such that
$\beta\alpha^{-1}=\delta_k\cdots\delta_1$
and for each $i,1\leq i< k$, we have:
\[|(\delta_i\cdots\delta_1)\cdot[W]|<|\alpha^{-1}\cdot [W]|\]
\end{lemma}

A factorization of $\beta\alpha^{-1}$ is \emph{peak-lowering} if it satisfies the conclusions of the lemma, so Lemma~\ref{le:machine} states that every peak has a peak-lowering factorization.
Such a factorization might not be peak-reduced, but the height of its highest peak is lower than the height of the peak in $\beta\alpha^{-1}$. 
We postpone the proof of Lemma~\ref{le:machine} to show how it implies \partref{it:lrwhalg}.

\begin{proof}[Proof of \partref{it:lrwhalg}]
Let $\gamma\in\langle\Omega_\ell\rangle$ and write
$\gamma=\alpha_k\cdots\alpha_1$
with $\alpha_1,\ldots\alpha_k\in\Omega_\ell$.
Let $h$ be
\[h=\sup \left\{|(\alpha_i\cdots\alpha_1)\cdot [W]|\big|\mbox{$\alpha_i$ is a peak} \right\}\]
which is the height of the highest peak in the factorization, 
and let $m$ be the number of maximal-height steps between peaks:
\[m=\Big|\left\{i\big|h=|(\alpha_i\cdots\alpha_1)\cdot [W]|\mbox{ and $\alpha_i$ is a peak or between two peaks}\right\}\Big|\]

If the factorization is not peak-reduced, then there is a peak $\alpha_i$ of maximal height.
Apply Lemma~\ref{le:machine} to the peak
$\alpha_{i+1}\alpha_i$ with respect to $(\alpha_{i-1}\cdots\gamma_1)\cdot W$ to get 
\[\alpha_{i+1}\alpha_i=\delta_j\cdots\delta_1\]
satisfying the conclusions of the lemma, and therefore a new factorization of $\gamma$:
\[\alpha_k\cdots\alpha_{i+2}\delta_j\cdots\delta_1\alpha_{i-1}\cdots\alpha_1\]

If $m$ of the old factorization was not 1, then $m$ of the new factorization is one less.
If $m$ of the old factorization was 1, then $h$ of the new factorization is strictly lower than $h$ of the old factorization.
By repeating this process, we eventually obtain a factorization with $h<1$.
This can only mean that there are no peaks, so we have a factorization which is peak-reduced.
\end{proof}

\begin{sublemma}
Let $\alpha$, $\beta$, and $[W]$ be as in Lemma~\ref{le:machine}. 
Then we have:
\begin{equation}\label{eq:pkineq}
2|\alpha^{-1}\cdot [W]|> |[W]|+|\beta\alpha^{-1}\cdot[W]|
\end{equation}
\end{sublemma}
\begin{proof}
Since $\beta\alpha^{-1}$ is a peak with respect to $[W]$, we can sum the two inequalities in the definition of a peak; by the fact that one of them is strict, we obtain this new inequality.
\end{proof}

\begin{sublemma}\label{sl:notnear}
Suppose we have $(A,a),(B,b)\in\Omega_\ell$ with $a\notin B$ and $\pg{a}$ not adjacent to $\pg{b}$ in $\Gamma$ (possibly $a=b^{-1}$).
Then $\lkl{a}\cap B=\emptyset$.
\end{sublemma}

\begin{proof}
If $x\in\lkl{a}\cap B$, then $x\in B$ and by Lemma~\ref{le:whdef}, either $b\geq x$ or $(B,b)$ acts on the connected component of $\pg{x}$ in $\Gamma-\st(\pg{b})$ by conjugation.
If the latter were true, since $\pg{a}$ is adjacent to $\pg{x}$ and not $\pg{b}$, we would have that $a\in B$, a contradiction.
So $b\geq x$, in which case $\pg{a}$ is adjacent to $\pg{b}$, a contradiction.
\end{proof}

\begin{sublemma}\label{sl:commshort}
Suppose $\alpha$, $\beta$, and $[W]$ are as in Lemma~\ref{le:machine},
and also that $\alpha=(A,a)$, $\beta=(B,b)$, and that either $\adj{a}{b}$ or that $A\cap B=\emptyset$ with $a^{-1}\notin B$.
Then
$
|\beta\cdot [W]|<|\alpha^{-1}\cdot [W]|
$.
\end{sublemma}

\begin{proof}
Take $W'$ to be a representative of $\alpha^{-1}\cdot [W]$ and take $W$ to be the obvious representative of $\alpha\cdot [W']$ based on $W'$ (this doesn't change our original $[W]$).
We will show the sublemma by analyzing adjacency counters.
First we claim that:
\begin{equation}\label{eq:twocounters1}\aco{B}{L-B}{W,b}\geq\aco{B}{L-B}{W',b}\end{equation}
We will show this by showing that every subsegment of $W'$ (meaning a subsegment of an element of $W'$) that is counted by the adjacency counter on the right above is also counted by the one on the left.
So suppose $(cud^{-1})^{\pm1}$ is counted on the right in Equation~(\ref{eq:twocounters1}), \emph{i.e.} $(cud^{-1})^{\pm1}$ is a subsegment of $W'$ with $c\in B$, $d\in L-B-\lkl{b}$, and with $u$ a word in $\lkl{b}$.

If $\adj{a}{b}$, then since $\beta\in\Omega_\ell$, we know $a^{\pm1}\notin B$.
This means that $\pg{c}\neq\pg{a}$.
Since $a\in\lkl{b}$, we also have $\pg{d}\neq\pg{a}$.
This means that the corresponding subsegment of $W$ is $(cu'd^{-1})^{\pm1}$, where $u'$ is $u$, possibly with instances of $a^{\pm1}$ added or removed.
This subsegment is then counted by the counter on the right.

If $A\cap B=\emptyset$, then $a\notin B$ and $b\notin A$.
Since $a^{\pm1}\notin B$, we know that $\pg{c}\neq\pg{a}$.
By Sublemma~\ref{sl:notnear}, we know $\lkl{b}\cap A=\emptyset$.
Note that it is possible that in passing to $W$, this letter $d$ could be deleted by a copy of $a^{-1}$ added to its left if $d=a^{-1}$.
We consider this case separately.

First suppose $d$ is not deleted.
Then the subsegment of $W$ corresponding to $(cud^{-1})^{\pm1}$ is either $(cud^{-1})^{\pm1}$ or $(cua^{-1}d^{-1})^{\pm1}$, depending on whether $d\in A$ or not.
In either case, this subsegment is counted once by the counter on the left, in the second case because $a\in L-B-\lkl{b}$.

If $d$ is deleted, then $d=a^{-1}$, and our original $(cud^{-1})^{\pm1}=(cua)^{\pm1}$ was part of a subsegment $(cuavf^{-1})^{\pm1}$, where $f\in (A-a)$ and $v$ is a word in $\lkl{a}$.
Note that by Sublemma~\ref{sl:notnear}, we know $\lkl{a}\cap B=\emptyset$, so the counter on the right counts this segment only once.
The subsegment of $W$ corresponding to this $(cuavf^{-1})^{\pm1}$ is then $(cuvf^{-1})^{\pm1}$.
Write $v=v'v''$ where $v'$ is the maximal initial segment of $v$ that is a word in $\lkl{b}\cap\lkl{a}$ and let $f'$ be the unique letter such that $cuv'(f')^{-1}$ is an initial segment of $cuvf^{-1}$.
Either $f'=f$ or $f'\in(\lkl{a}-\lkl{b})$.
In either case, $f'\in L-B-\lkl{b}$ (since $A\cap B=\emptyset$ and $B\cap\lkl{a}=\emptyset$), so the corresponding subsegment of $W$ is counted once by the counter on the left.
This shows our Equation~(\ref{eq:twocounters1}).

Now we will show:
\begin{equation}\label{eq:twocounters2}\aco{b}{L}{W',b}\geq\aco{b}{L}{W,b}\end{equation}
Suppose $(bud^{-1})^{\pm1}$ is a subsegment of $W$ counted by the counter on the right above, so $d\in L-\lkl{b}$ and $u$ is a word in $\lkl{b}$.
If $\adj{a}{b}$, then $W$ came from a subsegment $(bu'd^{-1})^{\pm1}$ of $W'$, where $u'$ is $u$, possibly with an instance of $a^{\pm1}$ added or removed; this subsegment is counted by the counter on the left.

If $A\cap B=\emptyset$, then either $d$ originally appeared in $W$ or it was added in passing to $W'$.
If it originally appeared in $W$, then $(bud^{-1})$ came from either a $(bud^{-1})^{\pm1}$ or a $(bu'au''d^{-1})^{\pm1}$, where in the second case $u''$ is the maximal terminal segment of $u$ that is a word in $\lkl{a}\cap\lkl{b}$; this subsegment of $W'$ is counted by the counter on the left.
If it was added, our $(bud^{-1})^{\pm1}$ in $W$ is part of a $(buavf^{-1})^{\pm1}$, with $v$ a word in $\lkl{a}$ and $f\in A$.
This subsegment is counted only once and came from a subsegment $(buvf^{-1})^{\pm1}$ that is counted once (for similar reasons as above).
This shows Equation~(\ref{eq:twocounters2}).

From Lemma~\ref{le:dico}, Equation~(\ref{eq:twocounters1}) and Equation~(\ref{eq:twocounters2}), we see that:
\[D_{[W]}(\beta)\geq D_{\alpha^{-1}\cdot[W]}(\beta)\]
By the definition of $D$, this means that 
$
|[W]|+|\beta\alpha^{-1}\cdot [W]|\geq |\alpha^{-1}\cdot [W]|+|\beta\cdot[W]|
$.
Combining this with Equation~(\ref{eq:pkineq}), we obtain
\[2|\alpha^{-1}\cdot [W]|>|\alpha^{-1}\cdot[W]|+|\beta\cdot[W]|,\]
which immediately implies the sublemma.
\end{proof}

\begin{proof}[Proof of Lemma~\ref{le:machine}]
If we have $\alpha=(A,a)$ and $\beta=(B,b)$, we will set $A'=L-A-\lkl{a}$ and $B'=L-B-\lkl{b}$.
Let $\overline\alpha=(A', a^{-1})$ and $\overline\beta=(B',b^{-1})$.
By Equation~(\ref{eq:R8}) and the fact that $\alpha$ and $\beta$ are long-range, these automorphisms describe the same elements of $\OAG$, and therefore $\alpha^{-1}\cdot [W]=\overline\alpha^{-1}\cdot [W]$ and $\beta\alpha^{-1}\cdot [W]=\overline\beta\alpha^{-1}\cdot [W]$.
We claim that if the lemma holds with $\alpha$ or $\beta$ replaced with $\overline\alpha=(A', a^{-1})$ or $\overline\beta=(B',b^{-1})$, respectively, then it holds as originally stated.
Suppose $\delta_k\cdots\delta_1$ is a peak-lowering factorization of $\overline\beta\alpha^{-1}$ (for example).
By Equation~(\ref{eq:R8}), the element $\beta\alpha^{-1}(\overline\beta\alpha^{-1})^{-1}$
is the conjugation $(L-b^{-1},b)$ (which is in $\Omega_\ell$).
If $|\beta\alpha^{-1}\cdot [W]|<|\alpha\cdot [W]|$ then 
\[\beta\alpha^{-1}=(L-b^{-1},b)\delta_k\cdots\delta_1\]
is a peak-lowering factorization of $\beta\alpha^{-1}$, since $(L-b^{-1},b)$ does not change the length of any conjugacy class.
Otherwise $|W|<|\alpha\cdot [W]|$.
Again by Equation~(\ref{eq:R8}), $\overline\beta\beta$ is the conjugation $(L-b,b^{-1})$.
So $(\overline\beta\alpha^{-1})^{-1}\beta\alpha^{-1}$ is $\alpha(L-b,b^{-1})\alpha^{-1}$.
If $b\notin A$, then by Equations~(\ref{eq:R9}) and~(\ref{eq:R10}), we know $(\overline\beta\alpha^{-1})^{-1}\beta\alpha^{-1}$ is a product of conjugations.
If $b\in A$, then by Equation~(\ref{eq:R8}), we know $(\overline\beta\alpha^{-1})^{-1}\beta\alpha^{-1}$ is $(L-a^{-1},a)\overline\alpha(L-b,b^{-1})\overline\alpha^{-1}(L-a,a^{-1})$, which is then a product of conjugations by Equations~(\ref{eq:R9}) and~(\ref{eq:R10}).
In any case, we have a product of conjugations $\gamma_j'\cdots\gamma_1'$ equal to $(\overline\beta\alpha^{-1})^{-1}\beta\alpha^{-1}$; then
\[\beta\alpha^{-1}=\delta_k\cdots\delta_1\gamma_j'\cdots\gamma_1'\]
is a peak-lowering factorization of $\beta\alpha^{-1}$, since conjugations do not change the length of conjugacy classes.
So we may swap out $\overline \alpha$ for $\alpha$ and $\overline\beta$ for $\beta$ as needed in the proof of this lemma.
Also, by the symmetry in the definition of a peak, we may switch $\alpha$ and $\beta$ if needed.

We fix a $k$-tuple of graphically reduced cyclic words $W$ representing the conjugacy class $[W]$.
Throughout this proof, $W'$ will denote the obvious representative of $\alpha^{-1}\cdot[W]$ based on $W$. 
We break this proof down into several cases.

\noindent
{\bf Case 1:} $\alpha$ is induced by a permutation of $L$.
Then $|\alpha\cdot [W]|=|[W]|$.
Since $(\beta,\alpha)$ is a peak, $\beta$ must shorten $\alpha\cdot[W]$, so $\beta=(B,b)$ for some $(B,b)$.  
Then the automorphism $(\alpha^{-1}(B),\alpha^{-1}(b))\in \Omega_\ell$ is well defined, and by Equation~(\ref{eq:R6}) the following factorization is peak-lowering:
\[\beta\alpha^{-1}=\alpha^{-1}(\alpha^{-1}(B),\alpha^{-1}(b))\]

In the remaining cases, we assume that $\alpha=(A,a)$ and $\beta=(B,b)$.
We will implicitly use Equation~(\ref{eq:R1}) to write $\alpha^{-1}$ as $(A-a+a^{-1},a^{-1})$ in the following.

\noindent
{\bf Case 2:} $a\in\lkl{b}$.
Of course, this implies that $\pg{a}\neq\pg{b}$.
Since $\alpha$ and $\beta$ are long-range, we know that $a,a^{-1}\notin B$ and $b, b^{-1}\notin A$.
Then by Equation~(\ref{eq:R3}b), we have:
\[\beta\alpha^{-1}=(B,b)(A-a+a^{-1},a^{-1})=(A-a+a^{-1},a^{-1})(B,b)=\alpha^{-1}\beta\]
By Sublemma~\ref{sl:commshort}, we know $|\beta\cdot [W]|<|\alpha^{-1}\cdot[W]|$, so this factorization is peak-lowering.

\noindent
{\bf Case 3:} $A\cap B=\emptyset$ and $a\notin\lkl{b}$.
We will break into sub-cases according to the configuration of $a^{-1}$ and $b^{-1}$.

\noindent
{\bf Sub-case 3a:} $\pg{a}=\pg{b}$.
Since $A\cap B=\emptyset$, this implies that $a=b^{-1}$.
By Equation~(\ref{eq:R2}), the following factorization is peak-lowering:
\[\beta\alpha^{-1}=(B,b)(A-a+b,b)=(A+B+b,b)\]

\noindent {\bf Sub-case 3b:} $a^{-1}\notin B$.
If $b^{-1}\notin A$, then
\[\beta\alpha^{-1}=(B,b)(A-a+a^{-1},a^{-1})=(A-a+a^{-1},a^{-1})(B,b)\]
by Equation~(\ref{eq:R3}a).
If $b^{-1}\in A$, then
by Equations~(\ref{eq:R2}) and~(\ref{eq:R4}a), we have:
\[\beta\alpha^{-1}=(B,b)(A-a+a^{-1},a^{-1})=(A+B-b-a+a^{-1},a^{-1})(B,b)\]
In either case, by Sublemma~\ref{sl:commshort}, $|\beta\cdot[W]|<|\alpha^{-1}\cdot[W]|$, so these factorizations are peak-lowering.

\noindent {\bf Sub-case 3c:} $\pg{a}\neq \pg{b}$, $a^{-1}\in B$, and $b^{-1}\in A$.
Note that since we are allowed to switch $\alpha$ and $\beta$, if $\pg{a}\neq\pg{b}$ and either $a^{-1}\notin B$ or $b^{-1}\notin A$, we are in sub-case 3b.
Therefore this sub-case finishes case 3.
Since $a^{-1}\in B$ and $a\notin B$, we see from Lemma~\ref{le:whdef} that $b\geq a$.
Similarly, $a\geq b$. 
So $a\sim b$ and by Lemma~\ref{le:whdef}, the automorphisms $\alpha'=(A,b^{-1})$ and $\beta'=(B,a^{-1})$ are well defined.

In the rest of this case, all adjacency counting is done with respect to $W'$.
Since $a\sim b$ and $\pg{a}$ is not adjacent to $\pg{b}$ in $\Gamma$, note that $\lkl{a}=\lkl{b}$ and therefore the adjacency counters with respect to $a$ and $b$ are the same functions.
Then by Lemma~\ref{le:dico}:
\[D(\alpha)+D(\beta)=D(\alpha')+D(\beta')\]
Also, by the definition of $D$ and Equation~(\ref{eq:pkineq}):
\[D(\alpha)+D(\beta)=-(2|\alpha^{-1}\cdot[W]|-|[W]|-|\beta\alpha^{-1}\cdot[W]|)<0\]
So either $D(\alpha')<0$ or $D(\beta')<0$.
Since we may swap $\alpha$ and $\beta$ (which swaps $\alpha'$ and $\beta'$), we assume $D(\beta')<0$.

Now we will find our peak-lowering factorization.
Let $\sigma_{a,b}$ be the type~(1) Whitehead automorphism from Equation~(\ref{eq:R5}).
By Equation~(\ref{eq:R5}), we have: 
\[\beta(\beta')^{-1}=(B,b)(B-a^{-1}+a,a)=(B-a^{-1}+a-b+b^{-1},a)\sigma_{a,b}\]
By Equation~(\ref{eq:R2}):
\[\beta'\alpha^{-1}=(B,a^{-1})(A-a+a^{-1},a^{-1})=(A+ B-a,a^{-1})\]
Then using
$\beta\alpha^{-1}=\beta(\beta')^{-1}\beta'\alpha^{-1}$, we have the factorization:
\[\beta\alpha^{-1}=(B-a^{-1}+a-b+b^{-1},a)\sigma_{a,b}(A+ B-a,a^{-1})\]
To show this factorization is peak-lowering, note the following:
\[
|(A+ B-a,a^{-1})\cdot[W]|
=|\beta'\alpha^{-1}\cdot[W]|
=D(\beta')+|\alpha^{-1}\cdot[W]|
<|\alpha^{-1}\cdot [W]|
\]
This is because $D(\beta')=|\beta'\alpha^{-1}\cdot[W]|-|\alpha^{-1}\cdot [W]|$.
Then since $\sigma_{a,b}$ does not change the length of a conjugacy class, this factorization is peak-lowering and we are done with this case.

\noindent {\bf Case 4:} $A\cap B\neq\emptyset$ and $a\notin\lkl{b}$.
All adjacency counting in this case is done with respect to $W'$.
First we show we can assume that we are in one of two sub-cases: either $a\notin B$ and $b\notin A$, or else $a\notin B$, $a^{-1}\in B$, $b\in A$ and $b^{-1}\notin A$.

Possibly by replacing $\beta$ with $\overline\beta$, we may assume that $a\notin B$.
Then if $b\notin A$, then we are in the first sub-case, so suppose $b\in A$.
First suppose $a^{-1}\in B$; if $b^{-1}\notin A$, then we are in the second sub-case, and if $b^{-1}\in A$, we can get to the first sub-case by swapping both $\alpha$ with $\overline\alpha$ and $\beta$ with $\overline\beta$.
Otherwise $a^{-1}\notin B$, and swapping $\alpha$ with $\overline\alpha$ puts us in the first sub-case.

In both of these sub-cases we will find that $\alpha^{-1}\cdot[W]$ is shortened by a well-defined Whitehead automorphism $(C\cap D,c)$, where $C$ is $A$ or $A'$, $D$ is $B$ or $B'$, and $c$ is an element of $\{a,a^{-1},b^{-1},b\}\cap A\cap C$.
By swapping $\alpha$ with $\overline\alpha$ if necessary, we assume $C=A$; similarly, we assume that $D=B'$.
Then $c$ is $a$ or $b^{-1}$, if it is $b^{-1}$, we swap $\alpha$ with $\beta$, $\alpha$ with $\overline\alpha$, and $\beta$ with $\overline\beta$ to get $(C\cap D,c)=(A\cap B',a)$.

Then in any event, we may assume that $(A\cap B',a)$ shortens $\alpha^{-1}\cdot[W]$.
We deduce from Lemma~\ref{le:whdef} that $(A-B'+a,a)$ is a well defined Whitehead automorphism.
From Equation~(\ref{eq:R2}) we have:
\[\alpha=(A- B'+a,a)(A\cap B',a)\]
Then we factor:
\[\beta\alpha^{-1}=\beta(A\cap B',a)^{-1}(A- B'+ a,a)^{-1}\]
Since $(A\cap B',a)$ shortens $\alpha^{-1}\cdot[W]$, we know that 
\[|(A-B'+a,a)^{-1}\cdot[W]|<|\alpha^{-1}\cdot[W]|\]
 and that $\beta(A\cap B',a)^{-1}$ is a peak with respect to $(A- B'+ a,a)^{-1}\cdot[W]$.
Then we can apply case 3 of this lemma to the peak $\beta(A\cap B' ,a)^{-1}$, and obtain a peak-lowering factorization of our original peak.

\noindent {\bf Sub-case 4a:} $a\notin B$, $a^{-1}\in B$, $b\in A$ and $b^{-1}\notin A$.
Then $a\sim b$ by Lemma~\ref{le:whdef}, and since $\pg{a}$ is not adjacent to $\pg{b}$, we have $\lkl{a}=\lkl{b}$.
Then adjacency counters taken with respect to $a$ and $b$ are the same.
Let $\gamma_1=(A\cap B,b)$, $\gamma_2=(A\cap B',a)$, $\gamma_3=(A'\cap B,a^{-1})$, and $\gamma_4=(A'\cap B', b^{-1})$.
Since $a\sim b$, these $\gamma_1,\gamma_2,\gamma_3$, and $\gamma_4$ are all well defined by Lemma~\ref{le:whdef}.

Now we will show that one of these automorphisms shortens $\alpha^{-1}\cdot[W]$.
Apply Lemma~\ref{le:dico} twice to get:
\[
D(\alpha)+D(\beta)=\aco{A}{A'}{a}-\aco{a}{L}{a}+\aco{B}{B'}{b}-\aco{b}{L}{b}\]
Then by further applications of Lemma~\ref{le:dico}, we obtain:
\[\begin{split}
\sum_{i=1}^4D(\gamma_i&)=\\
&\aco{A\cap B}{A'\cup B'}{b}-\aco{b}{L}{b}
+\aco{A'\cap B}{A\cup B'}{a^{-1}}-\aco{a^{-1}}{L}{a^{-1}}
\\&
+\aco{A\cap B'}{A'\cup B}{a}-\aco{a}{L}{a}
+\aco{A'\cap B'}{A\cup B}{b^{-1}}-\aco{b^{-1}}{L}{b^{-1}}\\
\end{split}\]
Putting these together, it follows from the additivity of adjacency counters that:
\[
2(D(\alpha)+D(\beta))=
\sum_{i=1}^4D(\gamma_i)+2(\aco{A\cap B}{A'\cap B'}{a}+\aco{A'\cap B}{A\cap B'}{a})
\]
Then by Equation~(\ref{eq:pkineq}):
\[0>2(D(\alpha)+D(\beta))\geq\sum_{i=1}^4D(\gamma_i)\]
This shows that for some $i$, $D(\gamma_i)<0$, so
one of $\gamma_1$, $\gamma_2$, $\gamma_3$ or $\gamma_4$ shortens $[W']$.
We have found an automorphism shortening $[W']$ as described above, so we are done with this sub-case.

\noindent{\bf Sub-case 4b:} $a\notin B$ and $b\notin A$.
We claim that $(A\cap B',a)$ is well defined.
By Lemma~\ref{le:whdef}, $(A\cap B',a)$ is well defined if for every $x\in A\cap B'$ with $a\not\geq x$, $(A\cap B',a)$ acts on component of $\pg{x}$ in $\Gamma-\st(\pg{a})$ by conjugation.
So suppose $x\in A\cap B'$ with $a\not\geq x$ and let $Y_1$ denote the component of $\pg{x}$ in $\Gamma-\st(\pg{a})$.
Then $(A\cap B',a)$ is well defined if for every $y\in Y_1$, we have $y,y^{-1}\in A$ and $y,y^{-1}\in B'$.
This first condition is true since $(A,a)$ acts on $Y_1$ by conjugation (since $a\not\geq x$).
So suppose for contradiction that $y\in B\cup \lkl{b}$ (meaning $y\notin B'$) with $\pg{y}\in Y_1$.
Then $y\in A$ and $a\not\geq y$.
By Sublemma~\ref{sl:notnear}, we know $A\cap\lkl{b}=\emptyset$.
This 
forces $y$ to be in $B$.
Let $Y_2$ be the component of $\pg{y}$ in $\Gamma-\st(\pg{b})$.
Since $y\in B$, either $b\geq y$ or $(B,b)$ conjugates $Y_2$.
If $b\geq y$, the fact that $\st(\pg{a})$ separates $\pg{b}$ from $\pg{y}$ means that $a\geq y$, a contradiction.
Then $(B,b)$ conjugates $Y_2$.
Since $a\notin B$, this means $a\notin Y_2$, which implies $\st(\pg{a})\cap Y_2=\emptyset$.
So since $\st(\pg{a})$ separates $\pg{b}$ from $\pg{y}$ in $\Gamma$, this means that $Y_2$ is also a component of $\Gamma-\st(\pg{a})$.
In that case, however, $Y_1=Y_2$, which implies $x\in B$, a contradiction.
So $(A\cap B',a)$ is well defined.
Note that $(B\cap A', b)$ is well defined by the same argument.

Next we will show that either $(A\cap B',a)$ or $(B\cap A', b)$ shortens $\alpha^{-1}\cdot[W]$.
By Equation~(\ref{eq:pkineq}), we know that $0>D(\alpha)+D(\beta)$.
By Lemma~\ref{le:dico}, we know that
\begin{align*}
D(\alpha)&=\aco{A}{A'}{a}-\aco{a}{L}{a}
=\aco{A\cap B'}{A'}{a}+\aco{A\cap B}{A'}{a}-\aco{a}{L}{a}
\end{align*}
and that:
\begin{align*}
D(\beta)&=\aco{B}{B'}{b}-\aco{b}{L}{a}
=\aco{B\cap A'}{B'}{b}+\aco{B\cap A}{B'}{b}-\aco{b}{L}{b}
\end{align*}
Also from Lemma~\ref{le:dico}, we know that
\begin{align*}
D((A\cap B',a))&=\aco{A\cap B'}{A'\cup B}{a}-\aco{a}{L}{a}\\
&=\aco{A\cap B'}{A'}{a}+\aco{B\cap A}{B'\cap A}{a}-\aco{a}{L}{a}
\end{align*}
and that:
\begin{align*}
D((B\cap A',b))&=\aco{B\cap A'}{B'\cup A}{b}-\aco{b}{L}{b}\\
&=\aco{B\cap A'}{B'}{b}+\aco{A\cap B}{A'\cap B}{b}-\aco{b}{L}{b}
\end{align*}

We claim that
$\aco{A\cap B}{A'}{a}\geq \aco{A\cap B}{A'\cap B}{b}$.
Since 
$b\notin A$, Sublemma~\ref{sl:notnear} says that 
$\lkl{b}\cap A=\emptyset$.
If $(cud^{-1})^{\pm1}$ is a subsegment of $W'$ with $c\in A\cap B$, $d\in A'\cap B$,  and $u$ a word in $\lkl{b}$, then either $u$ is a word in $\lkl{b}\cap\lkl{a}$, or $u=u'u_1u''$ where $u'$ a word in $\lkl{b}\cap\lkl{a}$ and $u_1\in \lkl{b}-\lkl{a}$.
If the former is true, $cud^{-1}$ is counted by $\aco{A\cap B}{A'}{a}$; if the latter is holds, then instead $cu'u_1$ is counted by $\aco{A\cap B}{A'}{a}$ (since $\lkl{b}\cap A=\emptyset$).
Either way, each subsegment of $W'$ counted by one counter is also counted by the other, showing the inequality.
Similarly, we know 
$\aco{B\cap A}{B'}{b}\geq \aco{B\cap A}{B'\cap A}{a}$.

Putting this all together, we have that:
\[0>D(\alpha)+D(\beta) >D((A\cap B',a))+D((B\cap A',b))\]
So one of $(A\cap B',a)$ and $(B\cap A',b)$ shortens $[W']$.
\end{proof}

\begin{remark}\label{re:purewhalg}
The \emph{pure automorphism group} $\AAGo$ of $A_\Gamma$ is the subgroup of $\AAG$ generated by dominated transvections, partial conjugations, and inversions.
It contains those graphic automorphisms which can be expressed as products of transvections and inversions; depending on $\Gamma$, $\AAGo$ may or may not be all of $\AAG$.
In any case, $\AAGo$ is a finite-index normal subgroup of $\AAG$.

Define the \emph{pure long-range Whitehead automorphisms} $\Omega_\ell^0$ to be $\Omega_\ell\cap \AAGo$.
If $\alpha\in\langle\Omega_\ell^0\rangle$, then in fact, we can peak reduce $\alpha$ with respect to any $k$-tuple of conjugacy classes $W$ by elements of $\Omega_\ell^0$.
To see this, consider the proof of Lemma~\ref{le:machine}: when we lower peaks in factorizations of $\alpha$, we move around type~(1) Whitehead automorphisms in case~1, and we introduce a type~(1) Whitehead automorphism in sub-case~3c that is in $\Omega_\ell^0$, and in no other case do we introduce a type~(1) Whitehead automorphism.
So if we start with a factorization of $\alpha$ by elements of $\Omega_\ell^0$ and peak-reduce it, we will end up with a peak-reduced factorization of $\alpha$ by elements of $\Omega_\ell^0$.
This technical detail is important for the application in Day~\cite{ssraag}.
\end{remark}


\section{Attempting to extend peak reduction to $\AAG$}
\subsection{A failure of peak-reduction}\label{se:nowhalg}
In this section we prove Proposition~\ref{mp:nowhalg}.

\begin{example}[Outer automorphisms of the four-vertex path]\label{ex:orbit}
Let $\Gamma$ be the four-vertex path, with labels as in Figure~\ref{fig:countex2}.
Let $P$ denote the subgroup of $\Out A_\Gamma$ generated by the images of the inversions and the single graphic automorphism (which swaps $a$ with $d$ and $b$ with $c$).
Then $P\cong (\Z/2\Z)\ltimes(\Z/2\Z)^4$.

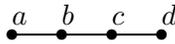
\begin{figure}[ht]
\begin{center}
\vspace{3em}
{
\setlength{\unitlength}{0.4 ex}
\begin{picture}(60,13)
\put(10,8){\circle*{2}}
\put(20,8){\circle*{2}}
\put(30,8){\circle*{2}}
\put(40,8){\circle*{2}}
\put(10,8){\line(1,0){30}}
\put(10,10){$a$}
\put(20,10){$b$}
\put(30,10){$c$}
\put(40,10){$d$}
\end{picture}
}
\end{center}
\caption{A graph $\Gamma$ such that peak-reduction fails on $A_\Gamma$.}
\label{fig:countex2}
\end{figure}

Let $N$ denote the subgroup of $\Out A_\Gamma$ generated by the images of the dominated transvections.
We have adjacent dominations $b\geq a$ and $c\geq d$, and non-adjacent dominations $c\geq a$ and $b\geq d$.
These are the only examples of domination in $\Gamma$.
This gives us six infinite cyclic subgroups of $\Aut A_\Gamma$ generated by dominated transvections: two for each example of non-adjacent domination (multiplying on the right and on the left) and one for each example of adjacent domination. 
Since $\adj{b}{c}$, these transvections commute and generate a copy of $\Z^6<\Aut A_\Gamma$.
Each of our pairs of non-adjacent transvections differ by an inner automorphism, so $N\cong\Z^4$.
From Equation~(\ref{eq:R6}), we know that $P$ normalizes $N$.
No vertex of $\Gamma$ has a star that separates $\Gamma$, so each partial conjugation is a full conjugation.
Then by Laurence's Theorem (Theorem~\ref{th:lau}), we have:
\[\Out A_\Gamma\cong P\ltimes N\]

Let $\phi\co\Z^4\to N$ be given by $\phi(p,q,r,s)(a)=ab^pc^q$ and $\phi(p,q,r,s)(d)=b^rc^sd$ (note $\phi(p,q,r,s)$ fixes the conjugacy classes $b$ and $c$).
Let $k\geq 2$ and let $w$ be the conjugacy class of the cyclic word $ad^k$.
For any $(p,q,r,s)\in \Z^4$, we have:
\[\phi(p,q,r,s)(w)=ab^{p+r}c^{q+ks}d(b^rd)^{k-1}\]
Note that the word on the right side is a graphically reduced cyclic word.
Then $\phi(p,q,r,s)$ fixes $w$ if and only if $p=0$, $r=0$, and $q=-ks$.
So the stabilizer $N_w$ is $\langle\phi(0,-k,0,1)\rangle$.
Further, the only classes in $(\Out A_\Gamma)\cdot w$ with length less than or equal to $|w|$ are the $8$ classes in $P\cdot w$.
Also note that if $w'=\sigma(w)$ with $\sigma\in P$, then the stabilizer $N_{w'}$ is $\langle \sigma\phi(0,-k,0,1)\sigma^{-1}\rangle$. 
\end{example}

\begin{proof}[Proof of Proposition~\ref{mp:nowhalg}]
Let $\Gamma$ be as in Example~\ref{ex:orbit}.
For $\alpha\in\Aut A_\Gamma$, let $|\alpha|$ denote the length of the class of $\alpha$ in $\Out A_\Gamma$ with respect to Laurence's generators.
Pick a natural number $k$ such that:
\[k > 1+\sup_{\alpha\in S}|\alpha|\]
Let $w$ be the conjugacy class of the cyclic word $ad^k$.
Let $\beta\in\Aut A_\Gamma$ represent the class of $\phi(0,-k,0,1)\in\Out A_\Gamma$, with $\phi$ as in Example~\ref{ex:orbit}.
Note that $\beta$ fixes $w$ and $\beta$ does not represent an element of $P$.

Suppose that $\beta$ can be peak reduced with respect to $w$ by elements of $S$.
Since $w$ is a minimal-length element of $(\Out A_\Gamma)\cdot w$, this means we can factor $\beta$ as $\gamma_m\cdots\gamma_1$ for some $\gamma_1,\ldots,\gamma_m\in S$, such that for each $j$, $1\leq j\leq m$, we have $|\gamma_j\cdots\gamma_1(w)|=|w|$.
Each $\gamma_{j}\cdots\gamma_1(w)$ is the same length as $w$ and in the same orbit, so by Example~\ref{ex:orbit}, it is in $P\cdot w$.

Fix a $j$, $1\leq j\leq m$.
Let $w'=\gamma_{j-1}\cdots\gamma_1(w)$.
There is $\sigma$ representing an element of $P$ such that $\sigma\gamma_j$ represents an element of $N$.
Since $|\sigma\gamma_j(w')|=|w'|$, we know from Example~\ref{ex:orbit} that $\sigma\gamma_j(w')=w'$ (since $w'\in P\cdot w$).
Also by Example~(\ref{ex:orbit}), $\sigma\gamma_j$ represents $\rho\phi(0,-ks,0,s)\rho^{-1}$ for some $s\in\Z$ and some $\rho\in P$. 
Then $|\sigma\gamma_j|=s(k+1)$; but since $\gamma_j\in S$ and therefore $|\sigma\gamma_j|\leq k$, this $s$ must be $0$.
Therefore each $\gamma_j$ represents an element of $P$.
Then $\beta$ represents an element of $P$, a contradiction.
\end{proof}

\subsection{Automorphisms fixing a set of basis elements}
In this section we prove Proposition~\ref{mp:specialwhalg}.

\begin{lemma}\label{le:shortensrgen}
Suppose $x$ is the conjugacy class of an element of $X$ and $\alpha\in\langle\Omega_s\rangle$.
Then $\alpha(x)$ cannot be shortened by a member of $\Omega_\ell$.
\end{lemma}

\begin{proof}
If an element of $\Omega_\ell$ shortens $\alpha(x)$, then it must be a type (2) automorphism $(A,a)$.
Further, we must have $a\in(\supp \alpha(x))^{\pm1}$ and $(A-a)\cap (\supp \alpha(x))^{\pm1}\neq\emptyset$ by Lemma~\ref{le:obvadj}.
We know $\supp \alpha(x)$ is a clique in $\Gamma$ by Lemma~\ref{le:srsupp}, so this contradicts the fact that $(A,a)$ is long-range.
\end{proof}

The image of $X$ in $H_\Gamma$ is a basis for $H_\Gamma$.
By declaring this basis to be orthonormal, we get an inner product $\langle-,-\rangle$ on $H_\Gamma$.
\begin{lemma}\label{le:knowdoms}
If $\alpha\in \langle\Omega_s\rangle$, then for any $a,b\in X$, we have
$\langle \alpha_*b,a\rangle\neq 0$ implies $a=b$, or $a\geq b$ with $\adj{a}{b}$.
\end{lemma}

\begin{proof}
We induct on the length of $\alpha$ with respect to the subset of transvections $\tau_{a,b}\in\Omega_s$.
The assertion is obvious if the length is zero.
Assume it is true for $\beta$ and that $\alpha=\beta\gamma$ where $\gamma=\tau_{c,d}\in\Omega_s$.
Suppose $\gamma=\tau_{c,d}$ for some $c,d\in L$ with $c\geq d$.
Then $\langle \alpha_*b,a\rangle\neq 0$ implies either that
$\langle \beta_*b,a\rangle\neq 0$ or that $\pg{a}=\pg{c}$ and $\langle \beta_*b,d\rangle\neq 0$.
In the first case, the lemma follows.
In the second case, we have $a\geq d$ and also $d\geq b$ with $\adj{d}{b}$ by inductive assumption.
\end{proof}

\begin{lemma}\label{le:fixgens}
Let $S\subset X$.
The pointwise stabilizer of $S$ in $\langle\Omega_s\rangle$ is generated by transvections $\tau_{a,b}\in\Omega_s$ with $\pg{b}\notin S$.
\end{lemma}

\begin{proof}
Suppose $\alpha\in\langle \Omega_s\rangle$ and $\alpha$ fixes $S$ pointwise.
Let $A=\alpha_*\in\Aut H_\Gamma$.
Since $\alpha$ fixes $S$ pointwise, for any $x\in S$, $\langle Ax,y\rangle=0$ for all $y\in X-x$ and $\langle Ax,x\rangle=1$.
Let $X=C_1\cup\cdots\cup C_m$ be the decomposition of $X$ into adjacent domination equivalence classes.
First of all, for each $i$, we can 
row-reduce $A$ such that for any $a,b\in C_i$, $\langle Ab,a\rangle$ is $0$ if $a\neq b$ and $1$ if $a=b$.
In fact, we can do this by multiplying $\alpha$ by transvections $\tau_{a,b}$ for various $a,b\in C_i$ with $\pg{b}\notin S$ (each $\tau_{a,b}$ corresponds to a row operation). 
Now suppose some $\langle A b,a\rangle\neq 0$ with $a\not\sim b$.
Then $a\geq b$ and $\adj{a}{b}$ by Lemma~\ref{le:knowdoms}, and $\pg{b}\notin S$.
Since we have already reduced the diagonal, applying some power of $\tau_{a,b}$ to $\alpha$ will change this entry to zero.
Of course, by doing this in appropriate order to the nonzero entries with $a\not\sim b$, we can row-reduce the rest of $A$.
So we can reduce $\alpha$ to the identity by applying elements $\tau_{a,b}\in\Omega_s$ with $\pg{b}\notin S$.
\end{proof}

\begin{proof}[Proof of Proposition~\ref{mp:specialwhalg}]
Suppose $W=(w_1,\ldots,w_k)$ is a $k$--tuple of conjugacy classes with each $|w_i|=1$ and suppose $\alpha\in\AAG$ with $|\alpha\cdot W|=|W|$.
By Theorem~\ref{mt:threeparts}, we write $\alpha=\beta\gamma$ where $\beta\in\langle\Omega_\ell\rangle$ and $\gamma\in\langle\Omega_s\rangle$.
Also by Theorem~\ref{mt:threeparts}, we have a factorization $\beta=\delta_m\cdots\delta_1$ by elements of $\Omega_\ell$ that is peak-reduced with respect to $\gamma\cdot W$.
By Lemma~\ref{le:shortensrgen}, this $\delta_1$ cannot shorten $\gamma\cdot W$.
So since $\delta_m\cdots\delta_1$ is peak-reduced, we have $|\gamma\cdot W|=|\alpha\cdot W|=|W|$.

Since each $w_i$ is a minimal-length representative of its $\AAG$--orbit, it follows that each $\gamma(w_i)$ is the conjugacy class of an element of $L$.
If $x,y\in L$ with $y$ conjugate to $\gamma(x)$, then $x\sim y$ by Corollary~\ref{co:tvgenstruct} and $e(x,y)$ by Lemma~\ref{le:srsupp}.
In general, if $S$ is a basis for $\Z^j$ for some $j$, $S'\subset S$ and $A\in\SL(j,\Z)$ sends $S'$ to a subset of $S^{\pm1}$, then there is $B\in\SL(j,\Z)$ such that $B|_{S'}=A|_{S'}$ and $B$ restricts to a permutation on $S\cup(-S)$ (this can be proven by a row reduction argument).
Then from Corollary~\ref{co:tvgenstruct}, we deduce that there is a type~(1) Whitehead automorphism $\sigma$ such that $\sigma\gamma\cdot W=W$ and $\sigma\in\langle\Omega_s\rangle$.

Then by Lemma~\ref{le:fixgens}, we can write $\sigma\gamma$ as a product $\phi_r\cdots\phi_1$ of elements $\phi_1,\ldots,\phi_r$ of $\Omega_s$ that fix $\supp W$ pointwise.
Then the following is a peak-reduced factorization of $\alpha$ by elements of $\Omega_\ell\cup\Omega_s$:
\[\alpha=\delta_m\cdots\delta_1\sigma^{-1}\phi_r\cdots\phi_1\]
\end{proof}

\begin{corollary}\label{co:fixvertsgenset}
Suppose $W=(w_1,\ldots,w_k)$ is a $k$--tuple of conjugacy classes of $A_\Gamma$ with each $|w_i|=1$.
Then the stabilizer $(\AAG)_W$ of $W$ in $\AAG$ is generated by $(\Omega_\ell\cup\Omega_s)\cap((\AAG)_W)$.
\end{corollary}

\begin{proof}
Let $\overline\Delta$ be the directed multi-graph whose vertices are $k$--tuples of conjugacy classes $W'$ with $|W'|=|W|$, and an edge from $W_1$ to $W_2$, labeled by $\alpha$, if $\alpha\in\Omega_\ell\cup\Omega_s$ with $\alpha(W_1)=W_2$.
Let $\Delta$ be the (undirected) connected component of $W$ in $\overline \Delta$.
This is called the \emph{Whitehead graph} of $W$.
We map the paths of $\Delta$ to $\AAG$ by composing their edge labels; a path from a vertex $W_1$ to a vertex $W_2$ will map to an automorphism $\alpha$ with $\alpha(W_1)=W_2$ (this is true for paths of length $1$ and remains true under concatenations).
In particular, $\pi_1(\Delta,W)\to(\AAG)_W$.
By Proposition~\ref{mp:specialwhalg}, if $\alpha\in(\AAG)_W$, we can write $\alpha=\beta_m\cdots\beta_1$ where each $\beta_i\in\Omega_\ell\cup\Omega_s$ and for each $i$, $0\leq i\leq m$, we have $|\beta_i\cdots\beta_1\cdot W|=|W|$.
Then $\beta_m\cdots\beta_1$ describes a path in $\Delta$ mapping to $\alpha$, and the map $\pi_1(\Delta,W)\to(\AAG)_W$ is surjective.

To get generators for $\pi_1(\Delta,W)$, we pick a maximal tree for $\Delta$.
Since each $|w_i|=1$, we know each vertex of $\Delta$ is the image of $W$ under some permutation of $L$.
Then we can pick our maximal tree $T$ to be a union of edges labeled by type~(1) Whitehead automorphisms originating at $W$.
There is a unique loop in $\pi_1(\Delta,W)$ for each (directed) edge in $\Delta-T$ (the loop leaving $T$ only to cross this edge once); these loops generate $\pi_1(\Delta,W)$, and the images of these loops in $(\AAG)_W$ generate.

If $\alpha$ is a type (2) Whitehead automorphism labeling an edge in $\Delta-T$, then $\alpha$ labels a loop from a vertex $W'$ to itself (if a type (2) Whitehead automorphism changes a vertex $W'$ of $\Delta$, then it lengthens it).
So if $W\neq W'$ there is a type~(1) Whitehead automorphism $\sigma$ labeling the edge in $T$ from $W$ to $W'$, and by Equation~(R6), the automorphism $\sigma \alpha\sigma^{-1}$ is a Whitehead automorphism labeling an edge from $W$ to itself.
If $\alpha$ is a type~(1) Whitehead automorphism labeling an edge in $\Delta-T$ that is not a loop at $W$, then by relations of type (R7), the loop based at $W$ through edges in $T$ and $\alpha$ is redundant with a type~(1) Whitehead automorphism labeling an edge from $W$ to itself.  
So in fact, the loops in $\pi_1(\Delta,W)$ of length $1$ map to a generating set for $(\AAG)_W$.
By definition, they map to $(\Omega_\ell\cup\Omega_s)\cap((\AAG)_W)$.
\end{proof}

\begin{remark}\label{re:pred}
There is another case where a peak-reduction theorem holds for $\AAG$: the author has shown in~\cite{ssraag} that if $w=[a_1,b_1]\cdots[a_k,b_k]$ for distinct $a_1,\cdots a_k$, $b_1,\cdots b_k\in X$, and $\alpha\in\AAG$ with $\alpha(w)=w$, then $\alpha$ can be peak reduced with respect to $w$ by elements of $\Omega$.
\end{remark}

\section{A presentation for $\AAG$}\label{se:pres}
The goal of this section is to prove Theorem~\ref{mt:pres}.
Recall that $\Phi$ is the free group on $\Omega$.
Let $\Phi_\ell<\Phi$ be the subgroup generated by $\Omega_\ell$.
Let $R_\ell=R\cap \Phi_\ell$.
%
%
%
Denote the normal closure of $\langle R_\ell\rangle$ in $\Phi_\ell$ by $\overline{\langle R_\ell\rangle}$.
Say that $w_1$ and $w_2$ in $\Phi_\ell$ are congruent modulo $R_\ell$ if $w_1w_2^{-1}\in \overline {\langle R_\ell\rangle}$.
Similarly, we denote the normal closure of $\langle R\rangle$ in $\Phi$ by $\overline{\langle R\rangle}$ and say that $w_1$ and $w_2$ in $\Phi$ are congruent modulo $R$ if $w_1w_2^{-1}\in\overline{\langle R\rangle}$.
%
\begin{lemma}\label{le:presmach}
Suppose $\alpha,\beta\in \Omega_\ell$ and $[W]$ is a $k$--tuple of conjugacy classes of $A_\Gamma$.
Suppose $\beta\alpha^{-1}$ forms a peak with respect to $[W]$.
Then there exist $\delta_1,\ldots,\delta_k\in \Omega_\ell$ such that, when multiplied in $\Phi_\ell$,
$\beta\alpha^{-1}$ is congruent to $\delta_k\cdots\delta_1$ modulo $R_\ell$
and for each $i,1\leq i< k$, we have:
\[|(\delta_i\cdots\delta_1)\cdot[W]|<|\alpha^{-1}\cdot [W]|\]
\end{lemma}

\begin{proof}
This lemma is a refinement of Lemma~\ref{le:machine}, so to prove it, it is enough to review the proof of Lemma~\ref{le:machine}, noting in each case that the peak-lowering factorization $\delta_k\cdots\delta_1$ is congruent to $\beta\alpha^{-1}$ modulo $R_\ell$.
This will be true if in each case, the only manipulations we apply to elements of $\AAG$ are applications of relations in $R$.
At the start of the proof, we established that if $\alpha=(A,a)$ and $\beta=(B,b)$, we may switch $\alpha$ and $\beta$ or swap $\beta$ with $(L-B-\lkl{b},b^{-1})$.
By the symmetry in the statement, it is again apparent that we may still switch $\alpha$ and $\beta$ if necessary.
In showing we could swap $\beta$ with $(L-B-\lkl{b},b^{-1})$, we used Relations~(\ref{eq:R8})--(\ref{eq:R10}).
In case 1, we used Relation~(\ref{eq:R6}).
We used Relation~(\ref{eq:R1}) in cases 2, 3 and 4.
In case 2, we used Relation~(\ref{eq:R3}b).
In case 3, we used Relation~(\ref{eq:R2}) in sub-case 3a; Relations~(\ref{eq:R2}),~(\ref{eq:R3}a) and~(\ref{eq:R4}a) in sub-case 3b; and Relations~(\ref{eq:R2}) and~(\ref{eq:R5}) in sub-case 3c.
In case 4, we used Relation~(\ref{eq:R2}) and invoked case 3.
These were the only manipulations done to elements of $\AAG$ in that proof, so we are done.  
\end{proof}

The following lemma is similar to Proposition~6.2.5 of Culler-Vogtmann.
\begin{lemma}\label{le:len2tuple}
Let $V$ be a $k$--tuple of conjugacy classes whose elements are all the conjugacy classes in $A_\Gamma$ of length $2$, each appearing once.
If $(A,a)\in\Omega_\ell$ and $|(A,a)\cdot V|\leq |V|$, then $(A,a)$ is trivial or is the conjugation $(L-a^{-1},a)$.
\end{lemma}

\begin{proof}
We partition $L$ into the following seven sets:
\begin{align*}
L=&(A\cap A^{-1})+ (A-A^{-1}-a)+ (A^{-1}-A-a^{-1})\\
&\quad + (L-\lkl{a}-A\cup A^{-1})+ \lkl{a}+ \{a\} + \{a^{-1}\}
\end{align*}
If $bc$ is a cyclic word of length $2$ (not necessarily with $b\neq c$), then 
we can use Lemma~\ref{le:obvadj} to compute $D_{[bc]}((A,a))$  according to the sets in the partition of $L$ that $b$ and $c$ are members of.
Note that since $bc$ is a cyclic word, we may switch $b$ with $c$ in our enumeration of cases.
Also note that if both $b,c\in (L-\lkl{a}-A\cup A^{-1})+ \lkl{a}+\{a\}+\{a^{-1}\}$, then $D_{[bc]}((A,a))=0$.
We list the remaining cases in Table~\ref{ta:DV}.
\begin{table}[ht!]
\begin{tabular}{|r|ccc|}
\hline
\backslashbox{$b$}{$c$} & $A\cap A^{-1}$ & $A-A^{-1}-a$ & $A^{-1}-A-a^{-1}$ \\
\hline 
$A\cap A^{-1}$ & 0 & & \\
$A-A^{-1} -a$ & 1 & 2 & \\
$A^{-1}-A-a^{-1}$ & 1 & 0 & 2 \\
$L-\lkl{a}-A\cup A^{-1} $ & 2 & 1 & 1\\
$\lkl{a}$ & 0 & 1 & 1\\
$\{a\}$ & 0 & 1 & -1 \\
$\{a^{-1}\}$ & 0 & -1 & 1 \\
\hline
\end{tabular}
\caption{The value of $D_{[bc]}((A,a))$ as $b$ and $c$ are in different subsets of $L$.}
\label{ta:DV}
\end{table} 

As usual, $n=|X|$.
Let $m=|X-\lk(\pg{a})|$, let $x=\frac{1}{2}|A\cap A^{-1}|$, and let $y=|A-A^{-1}-a|=|A^{-1}-A-a^{-1}|$.
Then $|L-\lkl{a}-A\cup A^{-1}|=2(m-x-y)$.
We list the number of conjugacy classes appearing in $V$ of the form $[bc]$ as $b$ and $c$ are in the different subsets of $L$ in Table~\ref{ta:Vsubs}, leaving out the cases in which $D_{[bc]}((A,a))=0$. 

\begin{table}[ht!]
\begin{tabular}{|r|ccc|}
\hline
\backslashbox{$b$}{$c$} & $A\cap A^{-1}$ & $A-A^{-1}-a$ & $A^{-1}-A-a^{-1}$ \\
\hline 
$A\cap A^{-1}$ & - & & \\
$A-A^{-1}-a $ & $2xy$ & $\frac{y(y+1)}{2}$ & \\
$A^{-1}-A-a^{-1}$ & $2xy$ & - & $\frac{y(y+1)}{2}$ \\
$L-\lkl{a}-A\cup A^{-1} $ & $4x(m-x-y)$ & $2y(m-x-y)$ & $2y(m-x-y)$\\
$\lkl{a}$ & - & $y(n-m)$ & $y(n-m)$\\
$\{a\}$ & - & $y$ & $y$ \\
$\{a^{-1}\}$ & - & $y$ & $y$ \\
\hline
\end{tabular}
\caption{The number of conjugacy classes in $V$ of the form $[bc]$, as $b$ and $c$ are in different subsets of $L$.}
\label{ta:Vsubs}
\end{table} 

We compute $D_V((A,a))$ from the two tables by taking products and summing:
\[D_V((A,a))=4xy+8x(m-x-y)+2y(y+1)+4y(m-x-y)+4y(n-m)\]
Note that the contribution to $D_V((A,a))$ from the entries in $V$ containing a copy of $a$ or $a^{-1}$ cancel each other out.
Since the numbers $x$, $y$, $(m-x-y)$, and $(n-m)$ are all nonnegative (they count the cardinalities of sets), we know that $D_V((A,a))$ cannot be negative; further, for $D_V((A,a))$ to be zero, we must have each of the terms equal to zero.
This implies that $y=0$, and that $x(m-x)=0$, which means that $(A,a)$ is either the trivial automorphism $(\{a\},a)$ or the conjugation $(L-\lkl{a}-a^{-1},a)$.
\end{proof}

\begin{lemma}\label{le:innisraag}
The group of inner automorphisms $\Inn A_\Gamma$ is a right-angled Artin group.
Specifically, if $Z$ is the intersection of the stars in $\Gamma$ of the elements of $X$, and $\Gamma'$ is the full subgraph of $\Gamma$ on the vertices $X-Z$, then the map sending $x\in X-Z$ to conjugation by $x$ in $A_\Gamma$ is an isomorphism $A_{\Gamma'}\isomarrow \Inn A_\Gamma$.
\end{lemma}

\begin{proof}
By the Servatius centralizer theorem (Theorem~\ref{th:centralizer}), we know that the center $Z(A_\Gamma)$ is $\langle Z\rangle <A_\Gamma$.
The obvious inclusion $A_{\Gamma'}\into A_\Gamma$ induces an isomorphism $A_{\Gamma'}\isomarrow A_\Gamma/Z(A_\Gamma)$; composing this map with the usual isomorphism $A_\Gamma/Z(A_\Gamma)\isomarrow\Inn A_\Gamma$ gives the isomorphism in the statement.
\end{proof}

The proof of the following proposition is based on McCool's argument from~\cite{mcpres}.
\begin{proposition}\label{pr:lrpres}
The group $\langle \Omega_\ell\rangle<\AAG$ has the presentation $\langle\Omega_\ell|R_\ell\rangle$.
\end{proposition}

\begin{proof}
We already know that every relation in $R_\ell$ is an identity of $\langle\Omega_\ell\rangle$, so it is enough to show that every word representing the trivial element in $\langle\Omega_\ell\rangle$ is a product of conjugates of elements of $R_\ell$.
Suppose $w\in\Phi_\ell$ represents the trivial element in $\langle\Omega_\ell\rangle <\AAG$.
We claim that there is $w'\in\overline{\langle R_\ell\rangle}$ such that $ww'$ is a product of type~(1) Whitehead automorphisms and conjugations.
Let $V_0$ be a $k$--tuple containing each conjugacy class of $A_\Gamma$ of length $2$ once.

We will prove the claim by induction on the peaks of $w$ with respect to $V_0$; specifically, inducting on the number of points between peaks of maximal height and also on the maximum height of peaks.
Write $w=\alpha_j\cdots\alpha_1$ for $\alpha_j,\ldots,\alpha_1\in\Omega_\ell$.
In our base case, we assume that $\alpha_j\cdots\alpha_1$ is a factorization of $w$ that is peak reduced with respect to $V_0$.
By Lemma~\ref{le:len2tuple}, we know that $V_0$ is a minimal-length representative of its $\AAG$ orbit.
So since our factorization of $w$ is peak reduced, for each $i$, we have $|(\alpha_i\cdots\alpha_1)\cdot V_0|=|V_0|$.
We claim that for each $i$, $(\alpha_i\cdots\alpha_i)\cdot V_0$ is a $k$--tuple containing each conjugacy class of length 2 once. 
This is true if $i=0$ by assumption.
Now assume it for $i-1$; since $|(\alpha_i\cdots\alpha_1)\cdot V_0|=|(\alpha_{i-1}\cdots\alpha_1)\cdot V_0|$, we know by Lemma~\ref{le:len2tuple} that $\alpha_i$ is then either trivial, a conjugation, or a type~(1) Whitehead automorphism and the statement is then true for $i$.
So in our base case, $w$ is already a product of type~(1) Whitehead automorphisms and conjugations.

For the inductive step, suppose that $\alpha_k\cdots\alpha_1$ has peaks with respect to $V_0$.
Let $\alpha_i$ be a peak of maximal height.
Then by Lemma~\ref{le:presmach}, there are $\delta_1,\ldots,\delta_m\in\Omega_\ell$ such that $(\alpha_{i+1}\alpha_i)^{-1}\delta_m\cdots\delta_1\in \overline{\langle R_\ell\rangle}$ and such that we can lower the peak at $\alpha_i$ in $\alpha_k\cdots\alpha_1$ by substituting in $\delta_m\cdots\delta_1$ for $\alpha_{i+1}\alpha_i$.
So we define:
\[w_1=(\alpha_{i-1}\cdots\alpha_1)^{-1}(\alpha_{i+1}\alpha_i)^{-1}\delta_m\cdots\delta_1(\alpha_{i-1}\cdots\alpha_1)\in \overline{\langle R_\ell \rangle}\]
Then $ww_1=\alpha_k\cdots\alpha_{i+2}\delta_m\cdots\delta_1\alpha_{i-1}\cdots\alpha_1$ has a smaller number of points between maximal-height peaks that $w$ with respect to $V_0$, or its maximal-height peak is shorter.
So we have reduced the peaks of $ww_1$, and we invoke the inductive hypothesis for $ww_1$: we have a $w_2\in\overline{\langle R_\ell\rangle}$ such that $ww_1w_2$ is a product of type~(1) Whitehead automorphisms and conjugations. 
So $w_1w_2\in\overline{\langle R_\ell \rangle}$ satisfies the conclusions of our inductive claim.

So we have that $w$ is congruent modulo $R_\ell$ to a product of type (1) Whitehead automorphisms and conjugations.
Then by applying instances of Relation~(\ref{eq:R6}), we know that $w$ is congruent to a product $\beta\gamma$ where $\beta$ is a product of type~(1) Whitehead automorphisms and $\gamma$ is a product of conjugation automorphisms in $\Omega_\ell$.
The subgroup of $\AAG$ generated by type~(1) Whitehead automorphisms acts faithfully on $\Aut H_\Gamma$, so since $\alpha$ maps to the trivial element of $\AAG$ and $\gamma$ is in the kernel of the homology representation, we deduce that $\beta$ represents the trivial automorphism.
So by some instances of Relation~(R7), we know that $w$ is congruent modulo $R_\ell$ to $\gamma$, which represents the trivial automorphism in $\Inn A_\Gamma$.

Let $Z$ and $\Gamma'$ be as in Lemma~\ref{le:innisraag}.
Map the free group on $X-Z$ to $\Phi_\ell$ by sending $a\in X-Z$ to $(L-\lkl{a}-a^{-1},a)$.
This sends the relations from the right-angled Artin group presentation of $A_{\Gamma'}$ to instances of Relation~(\ref{eq:R3}b).
Of course, this map descends to the isomorphism $A_{\Gamma'}\isomarrow \Inn A_\Gamma$ in Lemma~\ref{le:innisraag}.
Then since $\gamma$ represents the trivial element of $\Inn A_\Gamma$, it corresponds to an element $w_\gamma$ of the free group on $X-Z$ that maps to the trivial element of $A_{\Gamma'}$.
This $w_\gamma$ is a product of conjugates of the relations from the presentation of $A_\Gamma$, so $\gamma$ is a product of conjugates of instances of Relation~(\ref{eq:R3}b).
So $\gamma$ is in $\overline{\langle R_\ell\rangle}$, and therefore $w$ is in $\overline{\langle R_\ell \rangle}$.
\end{proof}

If $a,b\in X$ with $a\in\lkl{b}$ and $a\sim b$, then the type~(1) Whitehead automorphism $\sigma_{a,b}$ of Relation~(\ref{eq:R5}) exists.
According to that relation, we have $\sigma_{a,b}\in \langle\Omega_s\rangle$.
Let $P_s\subset \Omega$ be the finite subgroup of $\AAG$ generated by such $\sigma_{a,b}$ as $a$ and $b$ range over all adjacent domination-equivalent pairs in $X$.
Let $\Phi_s$ be the free subgroup of $\Phi$ generated by $\Omega_s\cup P_s$.
Let $R_s=R\cap \Phi_s$.
\begin{proposition}\label{pr:srwhpres}
The group $\langle \Omega_s\rangle$ has the presentation $\langle\Omega_s\cup P_s|R_s\rangle$.
\end{proposition}

\begin{proof}
Let $G=\langle\Omega_s\rangle<\AAG$
 and let $\tilde G=\langle \Omega_s\cup P_s|R_s\rangle$.
We know that each of the relations in $R_s$ is an identity in $\tilde G$, so we have homomorphism $G\to \tilde G$ by sending each element of $\Omega_s$ to its own coset.
We will show this map is an isomorphism by constructing an inverse.

By Corollary~\ref{co:tvgenstruct}, the group $G$ has a presentation where the generators are $\{E_{a,b}|\text{$a,b\in X$, $a\in\lkl{b}$, and $a\geq b$}\}$ and the relations are all the relations of the forms~(1)--(4) from Proposition~\ref{pr:linpres}.
This presentation identifies each $E_{a,b}$ with the corresponding $\tau_{a,b}$. 

By Relations~(\ref{eq:R1}), (\ref{eq:R2}), (\ref{eq:R5}) and~(R7), we know $\tilde G$ is generated by the transvections $(\{a,b\},a)$ with $a\in\lkl{b}$ and $a\geq b$.
We map $\tilde G$ to $G$ by sending each $(\{a,b\},a)$  to the corresponding $\tau_{a,b}$.
We will show that this is a homomorphism by checking the relations of 
our presentation for $G$ already hold in $\tilde G$.
Relation~(\ref{it:r1}) follows from Relations~(\ref{eq:R2}) and~(\ref{eq:R3}b).
Relation~(\ref{it:r2}) follows from Relation~(\ref{eq:R4}b).
For any $a,b\in X$ with $a\in\lkl{b}$ and $a\sim b$, we know from Relation~(\ref{eq:R5}) that $\tau_{a,b}\tau_{b,a}^{-1}\tau_{a,b}$ is $\sigma_{a,b}$, which has order 4 by Relation~(R7) (here we are using that $\tau_{a^{-1},b^{-1}}=\tau_{a,b}$, which holds because $a\in\lkl{b}$, and that $\tau_{a^{-1},b}=\tau_{a,b}^{-1}$).
This means that Relation~(\ref{it:r3}) already holds in $\tilde G$.
By Relation~(\ref{eq:R5}), $(\tau_{a,b}\tau_{b,a}^{-1}\tau_{a,b}\tau_{b,a})^3$ is $(\sigma_{a,b}\tau_{b,a})^3$, which is $\tau_{a,b^{-1}}\tau_{b^{-1},a^{-1}}\tau_{a^{-1},b}\sigma_{a,b}^3$ by Relation~(\ref{eq:R6}), which is $\sigma_{a,b}^{-1}\sigma_{a,b}^3$ by Relation~(\ref{eq:R5}) (and using the facts that $\tau_{a,b^{-1}}=\tau_{a,b}^{-1}$, $\tau_{b^{-1},a^{-1}}=\tau_{a^{-1},b}^{-1}$ and $\tau_{a^{-1},b}=\tau_{a^{-1},b^{-1}}^{-1}$).
Then Relation~(\ref{it:r4}) already holds in $\tilde G$.

So we map $\tilde G$ to $G$ homomorphically by sending $(\{a,b\},a)\in\Omega_s$ to $E_{a,b}$.
It is apparent (from looking at the action on generating sets) that this homomorphism is the inverse to the homomorphism $G$ to $\tilde G$ that sends each element of $\Omega_s\cup P_s$ to its own coset.
So $\langle \Omega_s\rangle<\AAG$ has the presentation $\langle \Omega_s\cup P_s|R_s\rangle$.
\end{proof}

\begin{proposition}\label{pr:pressort}
Every $w\in\Phi$ is congruent modulo $R$ to a product $uv$ for some $u\in\Phi_\ell$ and $v\in\Phi_s$.
\end{proposition}

\begin{proof}
This proposition is a refinement of \partref{it:complements}.
The only manipulations of elements of $\AAG$ done in that proof are through the sorting substitutions in Definition~\ref{de:sort}.
Each of the sorting substitutions comes from applications of relations from $R$, as in Lemma~\ref{le:validsubs}.
So the entire argument goes through for $\langle \Omega|R\rangle$.
\end{proof}

\begin{proof}[Proof of Theorem~\ref{mt:pres}]
We have already shown that all the relations in $R$ are identities of $\AAG$ (Proposition~\ref{pr:identities}), so it is enough to show that any element of $\Phi$ representing the trivial element of $\AAG$ is in $\overline{\langle R\rangle}$.
Let $w\in\Phi$ represent the trivial element of $\AAG$.
By Proposition~\ref{pr:pressort}, $w$ is congruent modulo $R$ to a product $uv$ for $u\in\Phi_\ell$ and $v\in \Phi_x$.
Let $[u]\in\langle\Omega_\ell\rangle$ and $[v]\in\langle\Omega_s\rangle$ denote the elements of $\AAG$ they represent.

Let $W_0$ be the elements of $X$ as an $n$--tuple of conjugacy classes.
Suppose that $[v]$ is not a type~(1) Whitehead automorphism; then $[v]$ sends $W_0$ to a strictly longer $n$--tuple.
By \partref{it:lrwhalg}, we peak reduce $[u]$ with respect to $[v]\cdot W_0$.
Since $[u][v]$ is trivial, $[u]$ sends $[v]\cdot W_0$ to $W_0$; since we have peak reduced $[u]$, the first automorphism $\alpha\in\Omega_\ell$ in our peak-reduced factorization of $[u]$ shortens $[v]\cdot W_0$.
However, this contradicts Lemma~\ref{le:shortensrgen}.

So $[v]$ is a type~(1) Whitehead automorphism, which we write as $\sigma$.
Then $w$ is congruent to $u\sigma^{-1}\sigma v$ modulo $R$.
From Proposition~\ref{pr:srwhpres}, we know that $u\sigma^{-1}$ is a product of conjugates of members of $R_s\subset R$, and from Proposition~\ref{pr:lrpres}, we know that $\sigma v$ is a product of conjugates of members of $R_\ell\subset R$.
So $w\in\overline{\langle R\rangle}$.
\end{proof}

\section{Closing Remarks}
The applications of peak reduction on $F_n$ mentioned in the introduction all suggest further applications of Theorem~\ref{mt:threeparts}.
Firstly, peak reduction can be used to get finite generation and finite presentation results for stabilizers of $k$--tuples of conjugacy classes in $\Aut F_n$, as in McCool~\cite{mcpres}.
Along these lines, the author has used Theorem~\ref{mt:threeparts} in~\cite{ssraag} to show that an analog of the mapping class group of a surface inside $\AAG$ is finitely generated.
Generally, one could obtain further results similar to Corollary~\ref{co:fixvertsgenset} by proving propositions similar to Proposition~\ref{mp:specialwhalg}, \emph{i.e.} finding additional special cases where peak reduction works on all of $\AAG$.

Peak reduction on the free group $F_n$ makes an algorithm possible that determines whether two $k$--tuples of conjugacy classes in $F_n$ are in the same $\Aut F_n$ orbit (and makes it possible to find an automorphism taking one to the other, if it exists).
Please see Lyndon--Schupp~\cite{ls}, Chapter~1, Proposition~4.19 for a description of this algorithm.
As for free abelian groups, row-reduction lets us transform $k$--tuples of elements of $\Z^n$ standard representatives of their $GL(n,\Z)$--orbits (and more carefully, to find an automorphism taking one to another if it exists).
So it seems natural to conjecture the existence of a similar algorithm for $\AAG$:
\begin{conjecture}\label{con:autalg}
There is an algorithm which, given $u,v\in A_\Gamma$, produces $\alpha\in\AAG$ with $\alpha(u)=v$, or determines in finite time that no such automorphism exists.
\end{conjecture}
Part~(\ref{it:lrwhalg}) of Theorem~\ref{mt:threeparts} easily implies such an algorithm if we are only considering $\alpha\in\langle\Omega_\ell\rangle$, and \partref{it:srinj} suggests a row-reduction approach if we are only considering $\alpha\in\langle\Omega_\ell\rangle$.
However, it is not clear how these methods could be extended to apply to all of $\AAG$. 
Proposition~\ref{mp:nowhalg} indicates that it will not be possible to produce the algorithm in Conjecture~\ref{con:autalg} by a direct generalization of the approach for free groups. 

Finally, it may be possible to use these algorithmic techniques to improve our understanding of spaces that $\AAG$ acts on.
As in Culler-Vogtmann~\cite{cullervogtmann}, it should be possible to use peak-reduction techniques to find paths in $\AAG$--spaces that behave nicely with respect to combinatorial Morse functions.
In particular, this should help us to better understand outer space of right-angled Artin groups, as defined in Charney--Crisp--Vogtmann~\cite{ccv} for triangle-free $\Gamma$.
For general $\Gamma$, certain spaces of isometric actions of $A_\Gamma$ on CAT(0) cubical complexes are $\AAG$--spaces.
Hopefully our techniques could lead to a better understanding of these spaces as well.

\bibliographystyle{amsplain}
\bibliography{fpraag}

\noindent
Dept. of Mathematics, California Institute of Technology\\
Pasadena, Ca 91125\\
E-mail: {\tt mattday@caltech.edu}
\medskip

\end{document}